\documentclass[11pt,reqno]{amsart}
\usepackage[margin=1in,letterpaper]{geometry}
\usepackage{graphicx}
\usepackage{amssymb}
\usepackage{amsthm}
\usepackage{mathrsfs}
\usepackage{stmaryrd}
\usepackage{accents}
\usepackage{enumitem}
\usepackage{hyperref}
\hypersetup{pdfstartview={XYZ null null 1.00}} 
\usepackage{bbm}
\usepackage{mathrsfs}
\usepackage{amsfonts,amsmath,amsthm,amssymb,stmaryrd,color}
\newtheorem{theorem}{Theorem}[section]

\newtheorem{lemma}[theorem]{Lemma}
\newtheorem{Proposition}[theorem]{Proposition}
\newtheorem{Remark}{Remark}

\numberwithin{equation}{section}
\usepackage{amsmath}
\allowdisplaybreaks[4]

\def\r3{\mathbb{R}^3}

\begin{document}
\title[instability and stability of a hyperbolic--parabolic model] {on instability and stability of a quasi--linear hyperbolic--parabolic model for vasculogenesis}
\author{Qing Chen}
\address{Qing Chen \newline School of Mathematics and Statistics\\
Xiamen University of Technology\\
Xiamen, Fujian 361024, China} \email{chenqing@xmut.edu.cn}

\author{Huaqiao Wang}
\address{Huaqiao Wang \newline College of Mathematics and Statistics, Chongqing University,
Chongqing 401331, China.} \email{wanghuaqiao@cqu.edu.cn}

\author{Guochun Wu}
\address{Guochun Wu \newline Fujian Province University Key Laboratory of Computational
Science, School of Mathematical Sciences, Huaqiao University,
Quanzhou 362021, China} \email{guochunwu@126.com}

\thanks{* Corresponding author:
wanghuaqiao@cqu.edu.cn} 

\begin{abstract}
 In this paper, we are concerned with the instability and stability of a quasi--linear hyperbolic--parabolic system modeling vascular networks. Under the assumption that the pressure satisfies $\frac{\nu P'(\bar\rho)}{\gamma \bar\rho} < \beta$, we first show that the steady--state is linear unstable (i.e., the linear solution grows in time in $L^2$) by constructing an unstable solution. Then based on the lower grow estimates on the solution to the linear system, we prove that the steady--state is nonlinear unstable in the sense of Hadamard. On the contrary, if the pressure satisfies
$\frac{\nu P'(\bar\rho)}{\gamma \bar\rho} > \beta$, we establish the global existence for small perturbations and the optimal convergent rates for all--order derivatives of the solution by slightly getting rid of the condition proposed in [Liu-Peng-Wang, SIAM J. MATH. ANAL 54:1313--1346, 2022].
\end{abstract}

\maketitle


\section{Introduction}\label{S1}
The mechanism of blood vessel formation has been investigated by many experiments in the past decades (see \cite{C}). The process of formation of a vascular network starting from randomly seeded cells can be accurately tracked by observing the migration and aggregation of cells. Experiments [11] tracking of individual trajectories show marked persistence in the direction, with a small random component superimposed. The motion is directed towards a zone of higher concentration of cells, suggesting that chemotactic factors are active. Under the basic assumption that persistence and chemotaxis are the key features which determine the size of the structure, Gamba et al. \cite{GACC} and Serini et al. \cite{SAGG} proposed a theoretical model which enables one to reproduce well both the observed percolative transition and the typical scale of observed vascular networks. The model appears to be rather successful in describing in vitro experiments, where all the parameters are under control and one can easily tune the cell density. In this paper, we are interested in the following system which can be numerically reproduced by a quasi--linear hyperbolic--parabolic model of vasculogenesis system proposed in \cite{GACC}:
\begin{equation}\label{1eq}
 \left\{\begin{array}{lll}
  \rho_t + {\rm div} (\rho {\bf u})  =0,
  \\
  (\rho {\bf u})_t + {\rm div}(\rho {\bf u}\otimes {\bf u})+\nabla P =
  -\alpha \rho {\bf u} + \beta \rho \nabla \phi,
  \\
  \phi_t = \mu \Delta \phi - \nu \phi +\gamma\rho,
 \end{array}\right.
\end{equation}
for $(x, t)\in \mathbb{R}^3 \times \mathbb{R}^+$ and with initial data
\begin{equation}\label{1id}
 (\rho, {\bf u}, \phi)(x,0) = (\rho_{0}, {\bf u}_{0}, \phi_0)(x) \rightarrow \left(\bar\rho, 0, \frac{\gamma}{\nu}\bar\rho\right) \text{ as } |x|\rightarrow \infty.
\end{equation}
The unknown functions $\rho$, ${\bf u} = (u_1, u_2, u_3)^T$ and $\phi$ represent the endothelial cell density, the cell velocity and the concentration of the chemoattractant secreted by cells respectively. $P= P(\rho)$ is a monotone pressure function accounting for the fact that closely packed cells resist to compression due to the impenetrability of cellular matter. The parameters $\alpha>0$ is a drag coefﬁcient and $\beta>0$ measures the strength of the cell response. The other three positive constants $\mu$, $\nu$ and $\gamma$ denote the diffusion coefﬁcient, the inverse of the characteristic degradation time of the chemotactic factor, and the rate of release respectively.

For more information about this model, we can refer to \cite{GACC, KGP} and references therein. In the present paper, we consider the instability and stability of the system \eqref{1eq}. We assume that $\bar\rho>0$ throughout the paper. Moreover, we assume that the pressure function $P = P(\rho)$ is a smooth function satisfying $P'(\rho)>0$ if $\rho >0$.

From the mathematical point of view, the system \eqref{1eq} is a hyperbolic--parabolic system. However, it is
non--trivial to apply directly the ideas used in the classical hyperbolic--parabolic system into the system \eqref{1eq} due to the production term $\gamma\rho$ in the diffusion equation \eqref{1eq}$_3$ for the chemical factor. To our best knowledge, there are few results so far for the hyperbolic--parabolic system modeling vascular networks in the literature. The mathematical analysis of the vasculogenesis system was
initiated by Kowalczyk, Gamba and Preziosi in \cite{KGP}, where an addition viscous term $\Delta u$ is supplied in the equation \eqref{1eq}$_2$ to introduce an energy dissipation mechanism, which can be thought to model the slowing down of cells in the proximity of network structures. They deduced a detailed linear stability analysis around a uniform distribution of the two dimensional cell density, with the aim of checking the potential for structure formation starting from initial data representing a continuum cellular monolayer. The model is linear unstable at low cell densities and stabilizes linearly at higher densities. Using compensated
compactness tools, Francesco and Donatelli \cite{FD} studied the diffusive relaxation limits of nonlinear systems of Euler type modeling chemotactic movement of cells toward Keller--Segel type systems. Concerned
with the global existence of solutions for the Cauchy problem \eqref{1eq}--\eqref{1id}, Russo and Sepe \cite{RS} established the global existence and asymptotic behavior of classical solutions if initial data are close to a nonvacuum equilibrium in $H^l(\mathbb{R}^3)$ with $l\geq3$. Recently, Liu et al. \cite{LPW} established the large--time profile of solutions to the Cauchy problem \eqref{1eq}--\eqref{1id}. They showed that if the initial perturbation is small in $H^l(\mathbb{R}^3)$ with $l\geq3$, that is,
\begin{equation}\nonumber
 \left\|\left(\rho_0 - \bar\rho, {\bf u}_0, \phi_0 - \frac{\gamma}{\nu}\bar\rho\right)\right\|_{H^l(\mathbb{R}^3)} +\|\nabla\phi_0\|_{H^l(\mathbb{R}^3)}
\end{equation}
is small, the global existence of the solutions can be obtained, and in addition if
\begin{equation}\nonumber
 \left\|\left(\rho_0 - \bar\rho, {\bf u}_0, \phi_0 - \frac{\gamma}{\nu}\bar\rho\right)\right\|_{L^1(\mathbb{R}^3)}
\end{equation}
is also small, the solutions of the Cauchy problem tend time--asymptotically to linear diffusion waves around the constant state with the rate
\begin{equation}\nonumber
 \|(\rho - \bar\rho, \phi - \frac{\gamma}{\nu}\bar\rho)\|_{L^q(\mathbb{R}^3)}
 \le C(1+t)^{-\frac32\left(1-\frac{1}{q}\right)},
 \quad
 \|{\bf u}\|_{L^q(\mathbb{R}^3)}
 \le C(1+t)^{-\frac32\left(1-\frac{1}{q}\right) -\frac12}.
\end{equation}

It should be noted that in \cite{KGP, LPW}, the stability of the constant state was obtained under an indispensable condition on the pressure $P= P(\rho)$:
\begin{equation}\label{co-beta-small}
 \frac{\nu P'(\bar\rho)}{\gamma \bar\rho} > \beta,
\end{equation}
which ensures the system \eqref{1eq} satisfying the requirement on the eigenvalues of the [SK] stability condition (see \cite{SK, XK}) as same as other hyperbolic--parabolic systems such as Navier--Stokes system. More exactly, by elaborate analysis on solution semigroup of the linear system \eqref{eq-li}, we can conclude that if the condition \eqref{co-beta-small} holds, then the real part of all the eigenvalues are negative for the low frequency, which indicates that the solution would converge in time. In fact, the authors \cite{LPW} investigated the global existence and derived the optimal $L^p-$rate on the solutions under small perturbation condition. However, the decay rates of the derivatives of the solutions are barely obtained because of the difficulties coming from the nonlinearity of the system. In this paper, we will pay attention to the optimal convergent estimates on the all--order derivatives of the solutions. Actually, we are more curious about the opposite situation, that is, the pressure $P= P(\rho)$ satisfies
\begin{equation}\label{co-beta-large}
 \frac{\nu P'(\bar\rho)}{\gamma \bar\rho} < \beta.
\end{equation}
Based on the spectral analysis, the solution to the linear system \eqref{eq-li} with the condition \eqref{co-beta-large} may more likely grow in time, which implies that the steady--state would be linear and nonlinear unstable. In our paper, we will give clear answers to these issues.

Before stating our results, we shall introduce some notations and conventions used throughout the paper. We employ $H^m(\mathbb{R}^3)$ to denote the usual Sobolev spaces with norm $\|\cdot \|_m$ for $m\geq 1$ and $L^p(\mathbb{R}^3)$ to denote the $L^p$ spaces with norm $\|\cdot\|_{L^p}$ for $1\le p \le +\infty$ respectively for convenience. Set a radial function $\psi \in C_0^\infty(\mathbb{R}_\xi^3)$ such that $\psi(\xi)=1$ while $|\xi|\le 1$ and $\psi(\xi) =0$ while $|\xi|\geq 2$. Define the low frequent part of $f$ by
\begin{equation}\label{f-L}
 f^L = \mathcal{F}^{-1}[\psi(\xi)\widehat{F}],
\end{equation}
and define the high frequent part of $f$ by
\begin{equation}\label{f-H}
 f^H = \mathcal{F}^{-1}[(1-\psi(\xi))\widehat{f}],
\end{equation}
then $f = f^L + f^H$ if the Fourier transform of $f$ exists. We will employ the notation $A \lesssim B $ to represent that $A \le CB$ for a universal positive constant $C$ depending only on the parameters coming from the problem. And $C_i~(i = 1, 2, \ldots, 6)$ will also denote some positive constants that only depends on the parameters of the problem.

Let us focus on the instability of the linear system:
\begin{equation}\label{eq-li}
 \left\{\begin{array}{lll}
  \rho_t + \bar\rho{\rm div} {\bf u}  =0,
  \\
  {\bf u}_t + \frac{P'(\bar\rho)}{\bar\rho} \nabla\rho +\alpha {\bf u} -\beta \nabla \phi =0,
  \\
  \phi_t = \mu \Delta \phi - \nu \phi +\gamma\rho.
 \end{array}\right.
\end{equation}
And our first result can be stated as:
\begin{theorem}[Linear instability]\label{1mainth}
 Assume that $\frac{\nu P'(\bar\rho)}{\gamma \bar\rho} < \beta$. Then there exists one positive constant $\Theta$ and for any $\bar\Theta \in \left(0, \frac\Theta2\right]$, the linear system \eqref{eq-li} admits unstable solution $(\rho^l_{\bar\Theta}, {\bf u}^l_{\bar\Theta}, \phi^l_{\bar\Theta})$ with the initial data $({\rho}^l_{0,\bar\Theta}, {\bf u}^l_{0,\bar\Theta}, {\phi}^l_{0,\bar\Theta})$ satisfying
 \begin{equation}\nonumber
  (\rho^l_{\bar\Theta} - \bar\rho, {\bf u}^l_{\bar\Theta}, \phi^l_{\bar\Theta} - \frac\gamma\nu \bar\rho )\in C^0(0, \infty; H^3(\mathbb{R}^3))
 \end{equation}
 and
 \begin{equation}\nonumber
 \left\{\begin{array}{lll}
  e^{(\Theta-\bar{\Theta})t}\|{\rho}^l_{0,\bar\Theta} - \bar\rho\|_{L^2} \le \|{\rho}^l_{\bar\Theta} - \bar\rho\|_{L^2}
  \le e^{\Theta t}\|{\rho}^l_{0,\bar\Theta} - \bar\rho\|_{L^2} , \vspace{1ex}
 \\
  e^{(\Theta-\bar{\Theta})t}\|{\bf u}^l_{0,\bar\Theta} \|_{L^2} \le \|{\bf u}^l_{\bar\Theta}\|_{L^2}
  \le e^{\Theta t}\|{\bf u}^l_{0,\bar\Theta}\|_{L^2} ,\vspace{1ex}
 \\
  e^{(\Theta-\bar{\Theta})t}\|{\phi}^l_{0,\bar\Theta} - \frac{\gamma}{\nu} \bar\rho\|_{L^2} \le \|{\phi}^l_{\bar\Theta} - \frac{\gamma}{\nu} \bar\rho\|_{L^2}
  \le e^{\Theta t}\|{\phi}^l_{0,\bar\Theta} - \frac{\gamma}{\nu} \bar\rho\|_{L^2}.
  \end{array}\right.
 \end{equation}
 Moreover, the initial data $({\rho}^l_{0,\bar\Theta},{\bf u}^l_{0,\bar\Theta}, {\phi}^l_{0,\bar\Theta})$ of the linear system \eqref{eq-li} depend on $\bar\Theta$ and satisfy
 \begin{equation}\nonumber
  \|{\rho}^l_{0,\bar\Theta} - \bar\rho\|_{L^2}\|{\bf u}^l_{0,\bar\Theta} \|_{L^2}\left\|{\phi}^l_{0,\bar\Theta} - \frac{\gamma}{\nu} \bar\rho \right\|_{L^2}>0.
 \end{equation}
\end{theorem}

\begin{Remark}
 Due to the differences of the systems, we are unable to construct the unstable solution to the linear system \eqref{eq-li} by applying the modified variational method as in \cite{GS, JT, JJ, JJW, WT}. However, by exploiting the delicate spetral analysis on the solution semigroup of the linear system \eqref{eq-li}, we succeed in constructing one unstable solution with the help of cut--off technique. What's more important, we obtain the lower grow estimates slightly less than the optimal grow estimates on the solution. This plays significant role on proving the instability of the nonlinear system \eqref{1eq}.
\end{Remark}

Based on Theorem \ref{1mainth}, the instability of the nonlinear system \eqref{1eq} can be described as:
\begin{theorem}[Nonlinear instability]\label{2mainth}
 Assume that $\frac{\nu P'(\bar\rho)}{\gamma \bar\rho} < \beta$. Then the steady state $(\bar\rho, 0, \frac{\gamma}{\nu}\bar\rho)$ of the system \eqref{1eq} is unstable  in the Hadamard sense, that is, there exist positive constants $\epsilon_0$ and $\delta_0$, such that for any $\delta \in (0, \delta_0)$,  and the initial data $(\rho_0, {\bf u}_0, \phi_0):= \delta ({\rho}^l_{0,\bar\Theta}, {\bf u}_{0,\bar\Theta}^l, {\phi}_{0,\bar\Theta}^l )$ with $({\rho}^l_{0,\bar\Theta}, {\bf u}^l_{0,\bar\Theta}, {\phi}^l_{0,\bar\Theta} )$ defined in Theorem \ref{1mainth} and the relevant parameter $\bar\Theta$ may depend on $\delta$, there is a unique strong solution $(\rho, {\bf u}, \phi)$ of the nonlinear system \eqref{1eq}, such that
 \begin{equation}\nonumber
  \|\rho(T^\delta) - \bar\rho\|_{L^2}, \|{\bf u}(T^\delta)\|_{L^2}, \|\phi(T^\delta)- \frac{\gamma}{\nu} \bar\rho\|_{L^2} \geq \epsilon_0
 \end{equation}
 for some escape time $T^\delta := \frac 1\Theta \ln\frac{2\epsilon_0}{\delta} \in (0, T^{max})$, where $T^{max}$ denotes the maximal time of existence of the solution $(\rho, {\bf u}, \phi)$.
\end{theorem}

Finally, we establish our result on the stability of the system \eqref{1eq} as:
\begin{theorem}[Global existence and Optimal decay rates]\label{3mainth}
 Assume that $\left(\rho_{0}-\bar\rho, {\bf u}_0, \phi_0 - \frac\gamma\nu \bar\rho\right) \in H^3(\mathbb{R}^3)$ and $\frac{\nu P'(\bar\rho)}{\gamma \bar\rho} > \beta$. If there exists a sufficiently small constant $\varepsilon_0>0$ such that
 \begin{equation}\label{id-delta0}
  \mathcal{E}_0
  = \left\|\left(\rho_{0}-\bar\rho, {\bf u}_0, \phi_0 - \frac\gamma\nu \bar\rho\right)\right\|_3 \le \delta_0,
 \end{equation}
 then the Cauchy problem \eqref{1eq}--\eqref{1id} admits a unique globally classical solution $(\rho, {\bf u}, \phi)$ satisfying
 \begin{equation}\nonumber
  \begin{split}
  &\left\|\left(\rho-\bar\rho, {\bf u}, \phi - \frac\gamma\nu \bar\rho\right)\right\|_3^2
  +\int_0^t \left(\|(\nu\phi -\gamma\rho)(\tau)\|_{L^2}^2 + \left\|\nabla \rho(\tau)\right\|_2^2 +
  \|{\bf u}(\tau)\|_3^2 +\|\nabla \phi(\tau)\|_3^2\right) d\tau
  \\
  &\le C\mathcal{E}_0^2 \text{  for any } t\in [0, \infty).
  \end{split}
 \end{equation}

 Under the assumption stated above, if in addition that
 \begin{equation}\label{k0}
  K_0
  =\left\|\left(\rho_{0}- \bar\rho, {\bf u}_0, \phi_0 - \frac\gamma\nu \bar\rho\right)\right\|_{L^1} <\delta_0,
 \end{equation}
 then for all $t\geq 0$ and for $0\le k\le 3$,  it holds that
 \begin{equation}\label{de-rho-phi}
  \left\|\nabla^k\left(\rho -\bar\rho, \phi - \frac\gamma\nu \bar\rho\right)\right\|_{L^2}
  \le C(1+t)^{-\frac34 -\frac k2},
 \end{equation}
 and
 \begin{equation}\label{de-u-low-high}
  \|\nabla^k{\bf u}\|_{L^2}
  \le C(1+t)^{-\frac54 -\frac k2}.
 \end{equation}
\end{theorem}

\begin{Remark}
 Under the assumption that $\left\|\left(\rho_{0}-\bar\rho, {\bf u}_0, \phi_0 - \frac\gamma\nu \bar\rho\right)\right\|_3$ is sufficiently small, the authors \cite{LPW} established the global existence of the solutions for the Cauchy problem \eqref{1eq}--\eqref{1id} and if in addition that $\left\|\left(\rho_{0}-\bar\rho, {\bf u}_0, \phi_0 - \frac\gamma\nu \bar\rho\right)\right\|_{L^1}$ is small enough, they derived the optimal $L^p-$rate on the solutions. In Theorem \ref{3mainth}, we can also obtain the global existence by getting rid of the assumption on the smallness of $\|\nabla^4\phi_0\|_{L^2}$. Moreover, we can get the optimal decay rates of all--order derivatives of the solutions.
\end{Remark}

Let us sketch the main ideas of proving Theorem \ref{1mainth}--\ref{3mainth} and illustrate the difficulties and challenges we need to deal with. Either for the stability or for the instability of the system \eqref{1eq}, it is crucial to analyze the semigroup of the linear system \eqref{eq-li}. Observed that there are fives equations in the system \eqref{1eq}, which means that it is much complicated to calculate the eigenvalues of the Fourier transform of the operator matrix
\begin{equation}\nonumber
 \left(\begin{array}{cccc}
 0 &   \bar\rho\nabla^T & 0
 \\
 \frac{P'(\bar\rho)}{\bar\rho}\nabla & \alpha \mathbb{I}_3 & -\beta\nabla
 \\
 -\gamma & {\bf 0}^T & -\mu\Delta +\nu
 \end{array}\right)
\end{equation}
from the linear system \eqref{eq-li}. our strategy is employing ``${\rm div}$--${\rm curl}"$ decomposition on the velocity ${\bf u}$, and then dividing the system \eqref{1eq} into two parts: one is a hyperbolic--parabolic system containing three equations (see \eqref{4eq-com-d}) and another one is an ODE (see \eqref{4eq-incom-omega}). Hence from the delicate analysis on the eigenvalues coming from the system \eqref{4eq-com-d}, we find that if the pressure satisfies $\frac{\nu P'(\bar\rho)}{\gamma \bar\rho} < \beta$, then the maximum of the real parts of the eigenvalues would strictly positive at the neighborhood of maximal point. This implies that the steady state may be unstable. Motivated by the idea in \cite{GS, JT, JJ, JJW, WT}, we show the instability of the linear system \eqref{eq-li} by constructing the initial data relying on the maximum of the real part of the eigenvalue. Ultimately, by using the lower grow estimates for the linear system and the decay estimate on the nonlinear terms, we prove the instability of the nonlinear system \eqref{1eq} in the Hadamard sense. On the contrary, if $\frac{\nu P'(\bar\rho)}{\gamma \bar\rho} > \beta$, the real parts of all the eigenvalues are strictly negative except at zero point, which reveals that the solution to the linear system \eqref{eq-li} will decay in time. Therefore, via the low--frequency and high--frequency decomposition on the solutions to the linear system, and combining with the estimates for the nonlinear term, we finally obtain the optimal convergent rates of all--order derivatives of the global solutions by making use of exhaustive spectral analysis and energy estimates.

The rest of this paper is organized as follows. In Section \ref{S2}, we do the spectral analysis on linearized system and show the linear instability. In Section \ref{S3}, we introduce some useful energy estimates on the system \eqref{1eq}. In Section \ref{S4}, we prove the nonlinear instability in the Hadamard sense. In Section \ref{S5}, we show the global existence and the optimal decay rates. In the last section, we recall the analytic tools.

\section{Spectral Analysis on Linearized System and Linear instability}\label{Spectral}\label{S2}
\subsection{Linearized System}
Let $a = \sqrt{P'(\bar\rho)}$ and $b = \frac{\bar\rho}{\sqrt{P'(\bar\rho)}} \beta$, then $\frac{\nu P'(\bar\rho)}{\gamma\bar\rho} -\beta = \frac{\sqrt{P'(\bar\rho)}}{\gamma\bar\rho}(a\nu - b\gamma)$. Take the linear transformation
\begin{equation}\nonumber
\left\{\begin{array}{lll}
 \varrho = \rho -\bar\rho,
 \\
 {\bf v} = \frac{\bar\rho}{\sqrt{P'(\bar\rho)}} {\bf u},
 \\
 \varphi = \phi - \frac\gamma\nu \bar\rho,
 \end{array}\right.
 \end{equation}
 and then the system \eqref{1eq} becomes to the following system:
\begin{equation}\label{3eq-lineari}
 \left\{\begin{array}{lll}
 \varrho_t + a{\rm div}{\bf v} = N_1,
 \\
 {\bf v}_t + a\nabla \varrho + \alpha {\bf v} - b\nabla \varphi = N_2,
 \\
 \varphi_t - \mu \Delta \varphi +\nu \varphi - \gamma \varrho = 0,
 \\
 (\varrho, {\bf v}, \varphi)(x,0) =(\rho_{0}- \bar\rho, \frac{\bar\rho}{\sqrt{P'(\bar\rho)}}{\bf u}_0, \phi_0 - \frac\gamma\nu \bar\rho)(x)
 \end{array}\right.
\end{equation}
with
\begin{equation}\label{3eq-lineari-N}
 \left\{\begin{array}{lll}
 N_1 = - \frac{\bar\rho}{\sqrt{P'(\bar\rho)}} {\rm div}(\varrho {\bf v}),
 \\
 N_2 = -\frac{\sqrt{P'(\bar\rho)}}{\bar\rho}{\bf v}\cdot\nabla {\bf v} - \frac{\bar\rho}{\sqrt{P'(\bar\rho)}} \left(\frac{P'(\varrho +\bar\rho)}{\varrho +\bar\rho} - \frac{P'(\bar\rho)}{\bar\rho} \right) \nabla\varrho.
 \end{array}\right.
\end{equation}

Let $d = \Lambda^{-1} {\rm div} {\bf v}$ be the ``compressible part" of ${\bf v}$ and $\Omega = \Lambda^{-1} {\rm curl} {\bf v}$ (with ${\rm curl}z = (\partial_{x_2}z^3 -\partial_{x_3}z^2, \partial_{x_3}z^1 -\partial_{x_1}z^3, \partial_{x_1}z^2 -\partial_{x_2}z^1)^T$) be the ``incompressible part" of ${\bf v}$ respectively. Here $\Lambda=\sqrt{-\Delta}$.  Then the system \eqref{3eq-lineari} can be rewritten as the following two parts:
\begin{equation}\label{4eq-com-d}
 \left\{\begin{array}{lll}
  \varrho_t + a\Lambda d = N_1,
  \\
  d_t - a\Lambda \varrho + \alpha d + b \Lambda \varphi = \Lambda^{-1}{\rm div}N_2,
  \\
  \varphi_t +\mu\Lambda^2\varphi +\nu\varphi - \gamma\varrho =0,
  \\
  (\varrho, d, \varphi)(x,0) = (\varrho_0, \Lambda^{-1}{\rm div}{\bf v}_0, \varphi_0)(x),
 \end{array}\right.
\end{equation}
and
\begin{equation}\label{4eq-incom-omega}
 \left\{\begin{array}{lll}
  \Omega_t + \alpha \Omega = \Lambda^{-1} {\rm curl} N_2,
  \\
  \Omega(x,0) = \Lambda^{-1}{\rm curl} {\bf v}_0(x).
 \end{array}\right.
\end{equation}

Note that \eqref{4eq-com-d} are hyperbolic--parabolic system that the structure of the solution semigroup is simpler than one of \eqref{3eq-lineari}, and \eqref{4eq-incom-omega} can be treated as an ODE for $\Omega$. Moreover, by the relationship
\begin{equation}\nonumber
 {\bf v} = -\Lambda^{-1}\nabla d - \Lambda^{-1}{\rm div}\Omega
\end{equation}
involving pseudo--differential operators of degree zero, the estimates in the space $H^3$ for the original function ${\bf v}$ can be derived from $d$ and $\Omega$. Hence we will focus on the spectral analysis on the solution semigroups of \eqref{4eq-com-d}--\eqref{4eq-incom-omega}.

\subsection{Spectral Analysis and Linear $L^2-$estimates}

Let $U = (\varrho, d, \varphi)^T$. Due to the semigroup theory for evolutionary equation, we will study the following initial value problem for the linear system:
\begin{equation}\label{5eq-semigp}
 \left\{\begin{array}{lll}
  U_t = \mathbb{B}U,
  \\
  U|_{t=0} = U_0 = (\varrho_0, d_{0}, \varphi_0)^T,
 \end{array}\right.
\end{equation}
where the operator $\mathbb{B}$ is given by
\begin{equation}\nonumber
\mathbb{B} = \left(\begin{array}{cccc}
 0 &   -a\Lambda & 0
 \\
 a\Lambda & -\alpha & -b\Lambda
 \\
 \gamma & 0 & -\mu\Lambda^2 -\nu
 \end{array}\right).
\end{equation}

Taking the Fourier transform to  the system, we have
\begin{equation}\label{f-u}
 \left\{\begin{array}{lll}
  \widehat{U}_t = \mathbb{A}(\xi)\widehat{U},
  \\
  \widehat{U}|_{t=0} = \widehat{U}_0 =(\widehat{\varrho_0}, \widehat{d_{0}}, \widehat{\varphi_0})^T,
 \end{array}\right.
\end{equation}
where $\widehat{U}(\xi, t) = \mathfrak{F}(U(x, t))$ and $\mathbb{A}(\xi)$ is given by
\begin{equation}\nonumber
 \mathbb{A}(\xi) = \left(\begin{array}{cccc}
 0 &   -a|\xi| & 0
 \\
 a|\xi| & -\alpha & -b|\xi|
 \\
 \gamma & 0 & -(\nu +\mu|\xi|^2)
 \end{array}\right).
\end{equation}

The eigenvalues of the matrix $\mathbb{A}(\xi)$ can be solved from the determinant
\begin{equation}\label{eq-eigenvalue}
 \begin{split}
  \det\{\mathbb{A}(\xi) - \lambda \mathbb{I}\}
  &= -\left(\lambda ^3 +(\mu|\xi|^2 +\alpha +\nu )\lambda^2 +\left(( a^2+\alpha\mu)|\xi|^2 +\alpha\nu\right)\lambda +a^2\mu|\xi|^4 +a(a\nu -b\gamma)|\xi|^2\right)
  \\
  &:= -F(|\xi|, \lambda) =0.
 \end{split}
\end{equation}

By direct calculation and delicate analysis on the roots of the above equation, we can deduce that the eigenvalues of the matrix $\mathbb{A}(\xi)$ has three different eigenvalues $\lambda_i = \lambda_i (\xi)$ with $i= 1, 2, 3$ while $|\xi|\ll1$ and $|\xi|\gg1$. Hence we can decompose the semigroup $e^{t\mathbb{A}(\xi)}$ as
\begin{equation}\nonumber
 e^{t\mathbb{A}(\xi)}
 = \sum_{i = 1}^3e^{\lambda_i t} \mathbb{P}_i(\xi)
\end{equation}
with the projector $\mathbb{P}_i(\xi)$ given by
\begin{equation}\label{P}
 \mathbb{P}_i(\xi)
 = \prod_{j\neq i}\frac{ \mathbb{A}(\xi) - \lambda_j I}{\lambda_i - \lambda_j},\quad i,j=1,2,3.
\end{equation}

Then we can represent the solution of the problem as
\begin{equation}\label{so-expr1}
 \widehat{U}(\xi, t)
 = e^{t\mathbb{A}(\xi)} \widehat{U}_0(\xi)
 = \left(\sum_{i=1}^3 e^{\lambda_i t}\mathbb{P}_i(\xi)\right) \widehat{U}_0(\xi).
\end{equation}

In order to derive large time properties of the semigroup $ e^{t\mathbb{A}(\xi)}$ in $L^2-$framework, we need to find out the asymptotical expansions of the eigenvalues and the projectors to analyze the semigroup. Hence by direct but tedious calculations on the cubic equation \eqref{eq-eigenvalue}, we have the following Taylor series expansions of the eigenvalues $\lambda_i~(i=1,2,3)$:
\begin{lemma}\label{eigenvalues}
 (i) There exists a positive constant $\eta_1\ll1$, such that for $|\xi|\le \eta_1$, the spectral has the following Taylor series expansion:
 \begin{equation}\nonumber
  \left\{\begin{array}{lll}\displaystyle
   \lambda_1
   = -\alpha + \frac{a^2(\alpha -\nu) +ab\gamma}{\alpha(\alpha-\nu)}|\xi|^2 +O(|\xi|^4),
   \\
   \displaystyle\lambda_2
   = -\nu - \frac{\mu\nu(\alpha - \nu) +ab\gamma}{\nu(\alpha -\nu)}|\xi|^2 +O(|\xi|^4),
   \\
   \displaystyle\lambda_3
   = -\frac{a(a\nu -b\gamma)}{\alpha\nu}|\xi|^2  +O(|\xi|^4),
  \end{array}\right.
  \end{equation}
  if $\alpha\neq \nu$, and
 \begin{equation}\nonumber
  \left\{\begin{array}{lll}\displaystyle
   \lambda_{1,2}
   = -\alpha \pm i \sqrt{\frac{ab\gamma}{\alpha}}|\xi| +\frac12\left(\frac{a(a\alpha - b\gamma)}{\alpha^2} -\mu\right)|\xi|^2 + O(|\xi|^3),
   \\
   \displaystyle\lambda_3
   = -\frac{a(a\alpha -b\gamma)}{\alpha^2}|\xi|^2  +O(|\xi|^4),
  \end{array}\right.
  \end{equation}
  if $\alpha = \nu$.

 (ii) There exists a positive constant $\eta_2\gg1$ such that, for $|\xi|\geq \eta_2$, the spectral has the following Taylor series expansion:
 \begin{equation}\nonumber
  \left\{\begin{array}{lll}\displaystyle
   \lambda_{1,2}
   =  \pm ia|\xi| -\frac\alpha 2 +O\left(\frac1{|\xi|}\right),
   \\
   \displaystyle\lambda_3
   = -\mu |\xi|^2 -\nu +O\left(\frac1{|\xi|^2}\right).
  \end{array}\right.
  \end{equation}
\end{lemma}

From Lemma \ref{eigenvalues} and the definition \eqref{P} of the project $\mathbb P_i$, we can obtain that if $|\xi|\le \eta_1$ and $\alpha \neq \nu$,
 {\small \begin{equation}\label{low-P1}
  \mathbb{P}_1(\xi)
  =\begingroup
   \renewcommand*{\arraystretch}{1.5}
   \left(\begin{array}{cccc}
   0 & \frac a\alpha |\xi| & 0
   \\
   \frac{b\gamma - a(\alpha - \nu)}{\alpha(\alpha - \nu)}|\xi| & 1 &\frac{b}{\alpha - \nu}|\xi|
   \\
   0 & -\frac{a\gamma}{\alpha(\alpha - \nu)}|\xi| & 0
   \end{array}\right)
  \endgroup + O(|\xi|^2)\mathbb{J}_3,
 \end{equation}}
{\small \begin{equation}\label{low-P2}
  \mathbb{P}_2(\xi)
  =\begingroup
   \renewcommand*{\arraystretch}{1.5}
   \left(\begin{array}{cccc}
   0 & 0 & 0
   \\
   \frac{b\gamma}{\nu(\alpha - \nu)} |\xi| & 0 & -\frac{b}{\alpha - \nu}|\xi|
   \\
   -\frac{\gamma}{\nu} & \frac{a\gamma}{\nu(\alpha -\nu)}|\xi| & 1
   \end{array}\right)
  \endgroup + O(|\xi|^2) \mathbb{J}_3,
 \end{equation}}
 and
{\small \begin{equation}\label{low-P3}
  \mathbb{P}_3(\xi)
  =\begingroup
   \renewcommand*{\arraystretch}{1.5}
   \left(\begin{array}{cccc}
   1 & -\frac a\alpha|\xi| & 0
   \\
  \frac{a\nu - b\gamma}{\alpha\nu}|\xi| & 0 & 0
   \\
   \frac \gamma\nu & -\frac{a\gamma}{\alpha\nu}|\xi| & 0
   \end{array}\right)
  \endgroup + O(|\xi|^2) \mathbb{J}_3,
 \end{equation}}
 where $\mathbb{J}_3$ is a 3--order matrix with all elements equal to 1, if $|\xi|\le \eta_1$ and $\alpha = \nu$,
 {\small\begin{equation}\nonumber
  \mathbb{P}_1(\xi)
  =\begingroup
   \renewcommand*{\arraystretch}{1.5}
   \left(\begin{array}{cccc}
   O(|\xi|) & O(|\xi|) & O(|\xi|)
   \\
   \frac1{2i}\sqrt{\frac{b\gamma}{a\alpha}} +O(|\xi|) & \frac12 +O(|\xi|) & -\frac1{2i}\sqrt{\frac{b\gamma}{a\alpha}}+O(|\xi|)
   \\
   -\frac{\gamma}{2\alpha} +O(|\xi|) & - \frac1{2i}\sqrt{\frac{b\gamma}{a\alpha}} +O(|\xi|) & \frac12 +O(|\xi|)
   \end{array}\right)
  \endgroup + O(|\xi|^2) \mathbb{J}_3,
 \end{equation}
 \begin{equation}\nonumber
  \mathbb{P}_2(\xi)
  =\begingroup
   \renewcommand*{\arraystretch}{1.5}
   \left(\begin{array}{cccc}
   O(|\xi|) & O(|\xi|) & O(|\xi|)
   \\
   - \frac1{2i}\sqrt{\frac{b\gamma}{a\alpha}} +O(|\xi|) & \frac12 +O(|\xi|) & - \frac1{2i}\sqrt{\frac{b\gamma}{a\alpha}} +O(|\xi|)
   \\
   -\frac{\gamma}{2\alpha} +O(|\xi|) &  \frac1{2i}\sqrt{\frac{b\gamma}{a\alpha}} +O(|\xi|) & \frac12 +O(|\xi|)
   \end{array}\right)
  \endgroup + O(|\xi|^2) \mathbb{J}_3,
 \end{equation}}
 \!\!and
 \begin{equation}\nonumber
  \mathbb{P}_3(\xi)
  =\begingroup
   \renewcommand*{\arraystretch}{1.5}
   \left(\begin{array}{cccc}
   1 & -\frac a\alpha|\xi| & 0
   \\
  \frac{a\alpha - b\gamma}{\alpha^2}|\xi| & 0 & 0
   \\
   \frac \gamma\alpha & -\frac{a\gamma}{\alpha^2}|\xi| & 0
   \end{array}\right)
  \endgroup + O(|\xi|^2) \mathbb{J}_3,
 \end{equation}
and if $|\xi|\geq \eta_2$,
 \begin{equation}\nonumber
  \mathbb{P}_1(\xi)
  =\begingroup
   \renewcommand*{\arraystretch}{1.5}
   \left(\begin{array}{cccc}
   \frac12 & -\frac1{2i} & 0
   \\
   \frac1{2i} & \frac12 &0
   \\
   0 & 0 & 0
   \end{array}\right)
  \endgroup + O\left(\frac1{|\xi|}\right) \mathbb{J}_3,
 \end{equation}

 \begin{equation}\nonumber
  \mathbb{P}_2(\xi)
  =\begingroup
   \renewcommand*{\arraystretch}{1.5}
   \left(\begin{array}{cccc}
   \frac12 & \frac1{2i} & 0
   \\
   -\frac1{2i} & \frac12 & 0
   \\
  0 & 0& 0
   \end{array}\right)
  \endgroup + O\left(\frac1{|\xi|}\right)  \mathbb{J}_3,
 \end{equation}
 and

 \begin{equation}\nonumber
  \mathbb{P}_3(\xi)
  =\begingroup
   \renewcommand*{\arraystretch}{1.5}
   \left(\begin{array}{cccc}
   0 & 0 & 0
   \\
   0 & 0 & 0
   \\
  0 & 0 & 1
   \end{array}\right)
  \endgroup + O\left(\frac1{|\xi|}\right)  \mathbb{J}_3.
 \end{equation}

In order to analyze the large time behavior of the semigroup $e^{t\mathbb{B}}$, we need to explore the middle frequent part of the eigenvalues and the projects while $a\nu -b\gamma> 0$ and $a\nu -b\gamma< 0$, respectively.
 \begin{lemma}\label{le-mi-ei}
Assume that $a\nu -b\gamma> 0$. Then for $\eta_1\le |\xi| \le \eta_2$, there exists a positive constant $\vartheta $ such that
  \begin{equation}\nonumber
   Re\lambda_i \le -\vartheta  \text{ and }  |\mathbb P_i|\le C,
  \end{equation}
  for $1\le i\le 3$.
 \end{lemma}
\begin{proof}
  Suppose $\lambda_i~(1\le i\le 3)$ are the eigenvalues of the matrix $\mathbb{A}(|\xi|)$. From \eqref{eq-eigenvalue} we have that $F(|\xi|, \lambda)$ is monotone increasing on $\lambda$, and $\lambda_i~(1\le i\le 3)$ satisfy
  \begin{equation}\label{weida}
   \left\{
   \begin{array}{lll}
    \lambda_1 +\lambda_2 +\lambda_3 = -(\mu|\xi|^2 +\alpha +\nu),
    \\
    \lambda_1\lambda_2 +\lambda_2\lambda_3 +\lambda_3\lambda_1 = (a^2 +\alpha\mu)|\xi|^2 +\alpha\nu,
    \\
    \lambda_1\lambda_2\lambda_3 = -(a^2\mu|\xi|^4 +a(a\nu -b\gamma)|\xi|^2).
   \end{array}\right.
  \end{equation}
  Furthermore, we can easily verify that for any $|\xi|>0$, $F(|\xi|, -(\mu|\xi|^2 +\alpha +\nu)) <0$ and $F(|\xi|, 0) >0$. This together with the monotonicity of $F$ implies that for any $|\xi|>0$,  if suppose $Im\lambda_1(|\xi|) = 0$, then $-(\mu|\xi|^2 +\alpha +\nu) <\lambda_1(|\xi|) <0$. Then from the first and third equalities of \eqref{weida}, we have that for $a\nu -b\gamma> 0$,
  \begin{equation}\nonumber
   \left\{
   \begin{array}{lll}
    \lambda_2 +\lambda_3 <0,
    \\
    \lambda_2\lambda_3 >0.
   \end{array}\right.
  \end{equation}
 Hence we can conclude that for any $|\xi|>0$, $Re\lambda_i(|\xi|) <0$. Moreover, by the continuity of $\lambda(|\xi|)$, we have that for any $\eta_1\le |\xi|\le \eta_2$, there exists a positive constant $\vartheta $ such that
$Re\lambda_i \le -\vartheta  $.

In addition, we can obtain the boundedness of the projectors via the definition \eqref{P} or the explicite expression of the projectors can be seen in \cite{LWWZ} if there exists multiple roots.
\end{proof}

 Next, we turn to investigate the case while $a\nu - b\gamma <0$. And we still suppose that $\lambda_i~(1\le i\le 3)$ are the eigenvalues of the matrix $\mathbb{A}(|\xi|)$.
 \begin{lemma}\label{xi0}
 Assume that $a\nu - b\gamma<0$. Then one of $Re\lambda_i~(1\le i\le 3)$ can achieves its supremum. More exactly, there exists one constant $\Theta>0$ and some $\xi_0 \neq 0$, such that
 \begin{equation}\label{Theta-sup}
  \Theta =
  \max_{i=1,2,3}\sup_{0<|\xi|<\infty}Re\lambda_i(|\xi|) = \max_{i=1,2,3}Re\lambda_i(|\xi_0|).
  \end{equation}
  Moreover, if we suppose that $\lambda_0(|\xi|)$ is the eigenvalue whose real part is equal to $\Theta$ at point $\xi_0$, then it holds
  \begin{equation}\nonumber
   Re\lambda_0(|\xi|) \geq Re\lambda_i(|\xi|), \quad i=1,2,3
  \end{equation}
  in some neighborhood of $\xi_0$.
\end{lemma}

\begin{proof}
Let ${\lambda}(|\xi|)$ be one of the eigenvalues of the matrix $\mathcal{A}(|\xi|)$, and while $|\xi|\le \eta_1$, ${\lambda}(|\xi|) = \lambda_3(|\xi|) = -\frac{a(a\nu -b\gamma)}{\alpha\nu}|\xi|^2 +O(|\xi|^4)$. Under the condition that $a\nu - b\gamma<0$, we can derive from Lemma \ref{eigenvalues} that
\begin{equation}\nonumber
 Re{\lambda}(\xi)
 \left\{\begin{array}{ll}
 >0, \quad |\xi|\ll 1,
 \\
 <0,\quad |\xi|\gg1,
 \end{array}\right.
\end{equation}
this combining with the continuity of ${\lambda}(|\xi|)$ on the interval $(0, \infty)$ and the fact that ${\lambda}(|\xi|)$ is monotone increasing in right neighborhood of ${\bf 0}$, implies that the function $Re{\lambda}(\xi)$ can achieves its supremum
 \begin{equation}\nonumber
  \sup_{ 0<|\xi|< \infty}Re{\lambda}(\xi),
  \end{equation}
 that is, there exists some $\xi_0$ satisfying $0 < |\xi_0|< \infty$ and some constant $\Theta > 0$, such that
 \begin{equation}\nonumber
  \Theta =\sup_{0 <|\xi|< \infty}Re\lambda(|\xi|) = Re\lambda(|\xi_0|).
  \end{equation}

 We claim that $\lambda(|\xi|) = \lambda_0(\xi)$. For any $\xi$ belonging to some small neighborhood of $\xi_0$, if $Im{\lambda}(|\xi|) =0$, then we can derive from the first and second equalities of \eqref{weida} that the real parts of the other two eigenvalues are negative; If $Im{\lambda}(|\xi|) \ne 0$, then $\lambda(|\xi|) +\bar{\lambda}(|\xi|) = 2 Re\lambda(|\xi|)>0$, which together of the first equality of \eqref{weida} gives the result that the remaining real eigenvalue is negative.
\end{proof}

With the help of  Lemma \ref{xi0}, we can get the upper boundedness of the real parts of the eigenvalues. And the boundedness of the projectors can be obtained as in Lemma \ref{le-mi-ei}. Hence we can conclude the estimates on the middle frequent part of the eigenvalues and the projectors as
 \begin{lemma}\label{le-mi-ei2}
 Assume that $a\nu -b\gamma< 0$. Then for $\eta_1\le |\xi| \le \eta_2$, it holds
  \begin{equation}\nonumber
   Re\lambda_i \le \Theta \ and \  |\mathbb P_i|\le C,
  \end{equation}
  for $1\le i\le 3$.
 \end{lemma}

 Finally, by combining with Lemma \ref{eigenvalues}, Lemma \ref{le-mi-ei} and Lemma \ref{le-mi-ei2}, we have the following $L^2-$estimate on the linear system \eqref{5eq-semigp}:
 \begin{Proposition}[L$^2$--theory]\label{pro-L2-theory}

 $(i)$ Assume that $a\nu -b\gamma>0$. Then it holds that
  \begin{equation}\label{L2-th1}
   \|e^{t\mathbb{A}(|\xi|)}\widehat{U_0}\|_{L^2(\eta_1\le |\xi|\le \eta_2)} \lesssim e^{-\vartheta t}\|U_0\|_{L^2}
   \end{equation}
  for any $t\geq0$.

 $(ii)$ Assume that $a\nu -b\gamma< 0$. Then it holds that
  \begin{equation}\label{L2-th2}
   \|e^{t\mathbb{A}(|\xi|)}\widehat{U_0}\|_{L^2(\eta_1\le |\xi|\le \eta_2)} \lesssim e^{\Theta t}\|U_0\|_{L^2}
   \end{equation}
  for any $t\geq0$.
  \end{Proposition}

\subsection{Linear instability}
In this subsection, we prove the linear instability of the system \eqref{eq-li}. Due to Lemma \ref{xi0} and the continuity of $\lambda_0(|\xi|)$, if $a\nu -a\gamma<0$, then for any $\bar{\Theta} \in \left(0, \frac{\Theta}{2}\right)$, there exists some positive constant $\displaystyle \bar{\zeta} = \bar{\zeta}(\bar{\Theta}) \le \frac{|\xi_0|}4$, such that
 \begin{equation}\nonumber
  \left|e^{t\lambda_0(\xi)}\right|
  \geq e^{(\Theta - \bar{\Theta})t}
 \end{equation}
holds for  $\xi$ satisfying $|\xi - \xi_0|\le 2\bar{\zeta}$.

Let $\Psi \in C_0^\infty(\mathbb{R}^3_\xi)$ be a radial function satisfying $\Psi(\xi) =1$ while $|\xi -\xi_0| < \bar{\zeta}$ and $\Psi(\xi) =0$ while $ |\xi -\xi_0| >2\bar{\zeta}$.
Then we set
\begin{equation}\label{li-varrho-u-varphi-1}
 \widehat{\varrho_{0,\bar\Theta}^l} = \frac{\Psi(\xi)}{\|\Psi\|_{L^2}}, \;\ \widehat{d_{0,\bar\Theta}^l} = -\frac{\lambda_0(|\xi|) \Psi(\xi)}{a|\xi|\|\Psi\|_{L^2}} \; \text{ and } \;
 \widehat{\varphi_{0,\bar\Theta}^l} = \frac{\gamma \Psi(\xi)}{\left(\lambda_0(|\xi|) + \nu + \mu|\xi|^2\right)\|\Psi\|_{L^2}}.
\end{equation}
We can easily check that $(\widehat{\varrho}^l_{\bar\Theta}, \widehat{d}^l_{\bar\Theta}, \widehat{\varphi}^l_{\bar\Theta}):= e^{\lambda_0(|\xi|)t}(\widehat{\varrho_{0,\bar\Theta}^l}, \widehat{d_{0,\bar\Theta}^l}, \widehat{\varphi_{0,\bar\Theta}^l})$ is a solution of the linear system \eqref{f-u} with the initial data $(\widehat{\varrho_{0,\bar\Theta}^l}, \widehat{d_{0,\bar\Theta}^l}, \widehat{\varphi_{0,\bar\Theta}^l})$.

Now we state the result in the following:
\begin{Proposition}[Linear instability]\label{li-insta2}
 Assume that $a\nu -b\gamma< 0$.
 Let
 \begin{equation}\nonumber
  \varrho^l_{\bar\Theta} = \mathcal{F}^{-1}\left[e^{\lambda_0(|\xi|)t}\widehat{\varrho_{0,\bar\Theta}^l}\right], \quad {\bf u}^l_{\bar\Theta} = \Lambda^{-1}\nabla\mathcal{F}^{-1}\left[e^{\lambda_0(|\xi|)t}\widehat{d_{0,\bar\Theta}^l}\right] \quad and \quad \varphi^l_{\bar\Theta} =  \mathcal{F}^{-1}\left[e^{\lambda_0(|\xi|)t}\widehat{\varphi_{0,\bar\Theta}^l}\right],
 \end{equation}
where $\widehat{\varrho_{0,\bar\Theta}^l}$, $\widehat{d_{0,\bar\Theta}^l}$ and $\widehat{\varphi_{0,\bar\Theta}^l}$ depending on $\xi_0$ and $\bar{\Theta}$ are defined in \eqref{li-varrho-u-varphi-1}. Then $(\varrho^l_{\bar\Theta}, {\bf u}^l_{\bar\Theta}, \varphi^l_{\bar\Theta})$ is a solution of the linear system \eqref{eq-li} with the initial data $(\varrho_{0,\bar\Theta}^l, {\bf u}_{0,\bar\Theta}^l, \varphi_{0,\bar\Theta}^l)$ depending on $\bar{\Theta}$. Moreover, for any $\bar{\Theta} \in \left(0, \frac{\Theta}{2}\right)$, we have
 \begin{equation}\label{es-li-instab1}
 \left\{\begin{array}{lll}
  e^{(\Theta-\bar{\Theta})t}\|{\varrho_{0,\bar\Theta}^l}\|_{L^2} \le \|\varrho^l_{\bar\Theta}\|_{L^2}
  \le e^{\Theta t}\|{\varrho_{0,\bar\Theta}^l}\|_{L^2} , \vspace{1ex}
 \\
  e^{(\Theta-\bar{\Theta})t}\|{{\bf u}_{0,\bar\Theta}^l} \|_{L^2} \le \|{\bf u}^l_{\bar\Theta}\|_{L^2}
  \le e^{\Theta t}\|{{\bf u}_{0,\bar\Theta}^l}\|_{L^2} , \vspace{1ex}
 \\
  e^{(\Theta-\bar{\Theta})t}\|{\varphi_{0,\bar\Theta}^l}\|_{L^2} \le \|\varphi^l_{\bar\Theta}\|_{L^2}
  \le e^{\Theta t}\|{\varphi_{0,\bar\Theta}^l}\|_{L^2}.
  \end{array}\right.
 \end{equation}
\end{Proposition}
 \begin{proof}
 By a direct calculation, \eqref{es-li-instab1} can be derived from the expression \eqref{li-varrho-u-varphi-1}  and the definition of $\Psi(\xi)$.
 \end{proof}

 Therefore, due to the equivalent relationship between the system \eqref{eq-li} and \eqref{3eq-lineari},  Theorem \ref{1mainth} can be easily derived from Proposition \ref{li-insta2}.

\subsection{Upper Decay Rate for the Linear system and the Nonlinear System}
By Lemma \ref{eigenvalues}, we can estimate the decay rates on the all--order derivatives of the solutions to the linear systems \eqref{4eq-com-d} and \eqref{4eq-incom-omega} while $N_1= N_2 = 0$ as follows.
\begin{lemma}\label{li-de-L-H} If $a\nu - b\gamma >0$, then we have
 \begin{equation}\label{es-varrho-low-k}
  \|\nabla^k\varrho^L(t)\|_{L^2}
   \le (1+t)^{-\frac34 -\frac k2} \|\varrho_0^L\|_{L^1} +(1+t)^{-\frac54 -\frac k2} \|d_0^L\|_{L^1} +(1+t)^{-\frac74 -\frac k2} \|\varphi_0^L\|_{L^1},
 \end{equation}
 \begin{equation}\label{es-d-low-k}
  \|\nabla^kd^L(t)\|_{L^2}
   \le (1+t)^{-\frac54 -\frac k2} \|\varrho_0^L\|_{L^1} +(1+t)^{-\frac74 -\frac k2} \|(d_0^L, \varphi_0^L)\|_{L^1} ,
 \end{equation}
 \begin{equation}\label{es-varphi-low-k}
\|\nabla^k\varphi^L(t)\|_{L^2}
   \le (1+t)^{-\frac34 -\frac k2} \|\varrho_0^L\|_{L^1} +(1+t)^{-\frac54 -\frac k2} \|d_0^L\|_{L^1}+(1+t)^{-\frac74 -\frac k2} \|\varphi_0^L\|_{L^1},
 \end{equation}
 \begin{equation}\label{es-varrho-d-varphi-high1}
  \|\nabla^k(\varrho^H, d^H, \varphi^H)(t)\|_{L^2}
  \le e^{-\frac\alpha 4t}\|\nabla^k(\varrho_0^H, d_0^H, \varphi_0^H)\|_{L^2},
 \end{equation}
 and
 \begin{equation}\label{es-Omega-all-k}
  \|\nabla^k\Omega(t)\|_{L^2}
  \le e^{-\alpha t}\|\nabla^k\Omega_0\|_{L^2}.
 \end{equation}
\end{lemma}
\begin{proof}
First from Lemma \ref{eigenvalues} and \eqref{low-P1}-- \eqref{low-P3} we can conclude that if $|\xi|\le \eta_1$ and $\alpha \ne \nu$,
 {\small\begin{equation}\label{so-max1}
  \begin{split}
  &\sum_{i = 1}^3e^{\lambda_i t} \mathbb{P}_i(\xi)
  =
  \\
  &\tiny{\begingroup
   \renewcommand*{\arraystretch}{1.8}
   \left(\begin{array}{cccc}
   e^{\lambda_3 t}  & \frac{a|\xi|}{\alpha}\left(e^{\lambda_1 t} - e^{\lambda_3 t}\right)& 0
   \\
   \frac{|\xi|}{ \alpha - \nu}\left(\frac{b\gamma - a(\alpha - \nu)}{\alpha} e^{\lambda_1 t} + \frac{b\gamma}{\nu} e^{\lambda_2 t}\right) + \frac{a\nu - b\gamma}{\alpha\nu}|\xi|e^{\lambda_3 t}& e^{\lambda_1 t} & \frac{b|\xi|}{\alpha - \nu}\left(e^{\lambda_1 t} - e^{\lambda_2 t}\right)
   \\
   \frac {\gamma}{\nu}\left(e^{\lambda_3 t} - e^{\lambda_2 t}\right) &  \frac{a\gamma|\xi|}{ \alpha - \nu}\left(\frac{1}{\nu}e^{\lambda_2 t} -\frac{1}{\alpha}e^{\lambda_1 t} \right) -\frac{a\gamma|\xi|}{\alpha\nu}e^{\lambda_3 t}& e^{\lambda_2 t}
   \end{array}\right)
   \endgroup}
   \\
   &+ O(|\xi|^2) \left(e^{\lambda_1 t} +e^{\lambda_2 t} +e^{\lambda_3 t}\right) \mathbb{J}_3.
  \end{split}
 \end{equation}}
 Then from \eqref{so-max1}, we have that for $|\xi|\le \eta_1$ and $\alpha \neq\nu$,
 \begin{equation}\nonumber
  \widehat{\varrho} = e^{\lambda_3 t} \widehat{\varrho_0} + \frac{a|\xi|}{\alpha}\left(e^{\lambda_1 t} - e^{\lambda_3 t}\right)\widehat{d_0} + O(|\xi|^2) \left(e^{\lambda_1 t} +e^{\lambda_2 t} +e^{\lambda_3 t}\right)(\widehat{\varrho_0} +\widehat{d_0} +\widehat{\varphi_0}).
 \end{equation}
 Therefore,  by using Plancherel theorem and Hausdorff--Young's inequality, we have from \eqref{so-expr1} and \eqref{L2-th1} that
  \begin{equation}\label{es-varrho-k-decay2}
   \begin{split}
   \|\nabla^k\varrho^L\|_{L^2}^2
   &= \int_{|\xi| <\infty}|\xi|^{2k}|\psi(\xi)\widehat{\varrho}(\xi)|^2 d\xi
   \\
   &\lesssim \int_{|\xi|<\eta_1}|\xi|^{2k}|\widehat{\varrho}(\xi)|^2 d\xi
   +\int_{\eta_1\le |\xi| \le2}|\xi|^{2k}|\psi(\xi)\widehat{\varrho}(\xi)|^2 d\xi
   \\
   &\lesssim \int_{|\xi|\le \eta_1} |\xi|^{2k}e^{2\lambda_3t}|\widehat{\varrho_0}|^2d\xi
   +\int_{|\xi|\le \eta_1} |\xi|^{2k}\left|\frac{a}{\alpha}|\xi|\left(e^{\lambda_1t} - e^{\lambda_3t}\right)\right|^2|\widehat{d_0}|^2d\xi
   \\
   &\quad +\int_{|\xi|\le \eta_1} |\xi|^{2k}\left|e^{\lambda_1t} +e^{\lambda_2t} +e^{\lambda_3t}\right|^2O(|\xi|^4)|(\widehat{\varrho_0}, \widehat{d_0}, \widehat{\varphi_0})|^2d\xi
   \\
   & \quad+\int_{\eta_1\le |\xi| \le2}e^{-2\vartheta t}|\psi(\xi)\widehat{U_0}(\xi)|^2 d\xi
   \\
   &\lesssim   \|\widehat{\varrho_0}\|_{L^\infty(|\xi|\le \eta_1)}^2 \int_{|\xi|\le \eta_1} |\xi|^{2k}e^{-\frac{a(a\nu -b\gamma)}{\alpha\nu}|\xi|^2 t}d\xi +\|\widehat{d_0}\|_{L^\infty(|\xi|\le \eta_1)}^2 \int_{|\xi|\le \eta_1} |\xi|^{2k+2}e^{-\frac{a(a\nu -b\gamma)}{\alpha\nu}|\xi|^2 t}d\xi
   \\
   &\quad +\|(\widehat{\varrho_0}, \widehat{d_0}, \widehat{\varphi_0})\|_{L^\infty(|\xi|\le \eta_1)}^2 \int_{|\xi|\le \eta_1} |\xi|^{2k+4}e^{-\frac{a(a\nu -b\gamma)}{\alpha\nu}|\xi|^2 t}d\xi
 +e^{-2\vartheta t}\|\psi\widehat{U_0}\|_{L^\infty}^2
   \\
   &\lesssim (1+t)^{-\frac32 -k}\|\varrho_0^L\|_{L^1}^2 +(1+t)^{-\frac52 -k}\|d_0^L\|_{L^1}^2 +(1+t)^{-\frac72 -k}\|\varphi_0^L\|_{L^1}^2 ,
   \end{split}
  \end{equation}
where $e^{-2\vartheta t}\lesssim (1+t)^{-\frac72 -k}$ and the definition \eqref{f-L} of $f^L$ are used. This implies that \eqref{es-varrho-low-k}--\eqref{es-varphi-low-k} holds for $\alpha \neq \nu$. And similarly we can conclude the results for the case $\alpha = \nu$. \eqref{es-varrho-d-varphi-high1} can be deduced by the similar argument in \eqref{es-varrho-k-decay2}, and \eqref{es-Omega-all-k} can be easily derived from the ODE for $ \Omega$.
\end{proof}

Now we turn to estimate the convergence rate for the nonlinear system \eqref{4eq-com-d}. To this end, we rewrite the system \eqref{4eq-com-d} as
\begin{equation}\label{5eq-noli}
 \left\{\begin{array}{lll}
  U_t = \mathbb{B}U + \mathcal{N},
  \\
  U|_{t=0} = U_0,
 \end{array}\right.
\end{equation}
with
\begin{equation}\nonumber
 \mathcal{N} = \mathcal{N} (\varrho, {\bf u})
 = (N_1, \Lambda^{-1}{\rm div}N_2, 0)^T.
\end{equation}
Then the solution of \eqref{5eq-noli} can be expressed as
\begin{equation}\label{so-ex}
 U =e^{tB}\ast U_0 + \int_0^t e^{(t-\tau)B}\ast \mathcal{N}(\tau)d\tau .
\end{equation}

Define
\begin{equation}\nonumber
 \mathcal{S}(x, \tau)
 := \left(S_1(x, \tau), S_2(x, \tau), 0\right)^T
 := e^{(t-\tau)B}\ast \mathcal{N}(\tau)
\end{equation}
and its Fourier transform
\begin{equation}\nonumber
 \mathfrak{F}[\mathcal{S}(x, \tau)]
 := (\widehat{S_1}(\xi, \tau), \widehat{S_2}(\xi, \tau), 0)^T.
\end{equation}

Now we complement the decay estimates on the nonlinear term of the expression \eqref{so-ex} of the solution $U(x, t)$ as follows, which can be derived from Lemma \ref{li-de-L-H}.
\begin{lemma}\label{prop-decay-nonli}
If $a\nu - b\gamma >0$, then we have
 \begin{equation}\label{decay-S1-low-k1}
  \|\nabla ^k(S_1^L)(\tau)\|_{L^2}
  \lesssim (1+t-\tau)^{- \frac34 -\frac k2}\|N_1^L(\tau)\|_{L^1} +(1+t-\tau)^{- \frac54 -\frac k2}\|N_2^L(\tau)\|_{L^1},
 \end{equation}
 \begin{equation}\label{decay-S1-low-k2}
  \|\nabla ^kS_1^L(\tau)\|_{L^2}
  \lesssim (1+t-\tau)^{- \frac34}\|\nabla^kN_1^L(\tau)\|_{L^1} +(1+t-\tau)^{- \frac54}\|N_2^L(\tau)\|_{L^1},
 \end{equation}
 \begin{equation}\label{decay-S2-low-k1}
  \|\nabla ^kS_2^L(\tau)\|_{L^2}
  \lesssim (1+t-\tau)^{- \frac54 -\frac k2}\|N_1^L(\tau)\|_{L^1} +(1+t-\tau)^{- \frac74 -\frac k2}\|N_2^L(\tau)\|_{L^1},
 \end{equation}
 \begin{equation}\label{decay-S2-low-k2}
  \|\nabla ^kS_2^L(\tau)\|_{L^2}
  \lesssim (1+t-\tau)^{- \frac54}\|\nabla^kN_1^L(\tau)\|_{L^1} +(1+t-\tau)^{- \frac74}\|\nabla^kN_2^L(\tau)\|_{L^1},
 \end{equation}
 and
 \begin{equation}\label{decay-S1-S2-high-k}
  \|\nabla ^k(S_1^H, S_2^H)(\tau)\|_{L^2}
  \lesssim e^{- \frac\alpha 4(t-\tau)}\|\nabla^k(N_1^H, N_2^H)(\tau)\|_{L^2} .
 \end{equation}
\end{lemma}

\section{Some useful estimates}\label{S3}
In this section, we give some useful estimates of the solutions.
First, we rewrite the Cauchy problem \eqref{1eq}--\eqref{1id} as
\begin{equation}\label{2eq-var}
 \left\{\begin{array}{lll}
 \varrho_t + (\varrho + \bar\rho){\rm div}{\bf u} = - {\bf u}\cdot \nabla\varrho,
 \\
 {\bf u}_t + \frac{P'(\varrho + \bar\rho)}{\varrho + \bar\rho} \nabla\varrho +\alpha {\bf u} - \beta \nabla \varphi = -{\bf u}\cdot \nabla {\bf u},
 \\
 \varphi_t - \mu \Delta \varphi +\nu \varphi - \gamma \varrho = 0,
 \\
 (\varrho, {\bf u}, \varphi)(x,0) =(\rho_{0}- \bar\rho, {\bf u}_0, \phi_0 - \frac\gamma\nu \bar\rho)(x).
 \end{array}\right.
\end{equation}
Then we make a priori assumption that
\begin{equation}\label{E-t-priori}
 \mathcal{E}(t) := \|(\varrho, {\bf u}, \varphi)\|_3 \le \delta \ll 1, \text{ for any } t\geq0.
\end{equation}

In order to obtain the a priori estimates on the solutions, we first introduce some useful estimates including the energy estimates on the solution and its derivatives.
\begin{lemma}\label{a priori-sol}
 Assume that the a priori assumption \eqref{E-t-priori} holds, then we have the following estimates:

 (i)
  \begin{equation}\label{es-varrho-0}
  \frac 12\frac{d}{dt}\|\varrho\|_{L^2}^2 - \bar\rho\langle {\bf u}, \nabla\varrho\rangle
  \lesssim \mathcal{E}(t) \left(\|\nabla\varrho\|_{L^2}^2 +\|\nabla {\bf u}\|_{L^2}^2\right),
 \end{equation}

 \begin{equation}\label{es-u-0}
  \frac12\frac{d}{dt}\|{\bf u}\|_{L^2}^2 + \frac{3\alpha}4\|{\bf u}\|_{L^2}^2 + \left\langle \frac{P'(\varrho +\bar\rho)} {\varrho +\bar\rho}\nabla \varrho, {\bf u}\right\rangle - \beta\langle\nabla\varphi, {\bf u}\rangle
 \le 0,
 \end{equation}
 and
 \begin{equation}\label{es-varphi-0}
  \frac{d}{dt}\|\nu\varphi - \gamma\varrho\|_{L^2}^2 + \mu\nu^2\|\nabla\varphi\|_{L^2}^2 + \nu\|\nu\varphi - \gamma\varrho\|_{L^2}^2
 \le 3\mu\gamma^2\|\nabla\varrho\|_{L^2}^2 +\frac{5\gamma^2\bar\rho^2}{2\nu}\|{\rm div}{\bf u}\|_{L^2}^2.
 \end{equation}

 (ii) For $k=1,2$,
 \begin{equation}\label{es-varrho-k-const}
  \frac12\frac{d}{dt}\|\nabla^k \varrho\|_{L^2}^2 - \bar\rho\left\langle \nabla^k {\bf u}, \nabla^{k+1}\varrho\right\rangle
 \lesssim \mathcal{E}(t)\left(\|\nabla^k \varrho\|_{L^2}^2 + \|\nabla^{k+1} {\bf u}\|_{L^2}^2\right).
 \end{equation}

(iii) For $k=1,2,3$
 \begin{equation}\label{es-varrho-k}
  \frac{d}{dt}\left\langle \frac {P'(\varrho +\bar\rho)}{2(\varrho +\bar\rho)^2}\nabla^k\varrho, \nabla^k\varrho\right\rangle - \left\langle  \frac{P'(\varrho +\bar\rho)} {\varrho +\bar\rho}\nabla^k {\bf u}, \nabla^{k+1}\varrho\right\rangle
 \lesssim \mathcal{E}(t)\left(\|\nabla^k \varrho\|_{L^2}^2 + \|\nabla^k {\bf u}\|_{L^2}^2\right),
 \end{equation}

  \begin{equation}\label{es-u-k}
  \frac12\frac{d}{dt}\|\nabla^k {\bf u}\|_{L^2}^2 + \frac{3\alpha}4\|\nabla^k {\bf u}\|_{L^2}^2 + \left\langle \frac{P'(\varrho +\bar\rho)} {\varrho +\bar\rho}\nabla^{k+1} \varrho, \nabla^k {\bf u}\right\rangle - \beta\langle\nabla^{k+1}\varphi, \nabla^k{\bf u}\rangle
 \lesssim \mathcal{E}(t)\|\nabla^k\varrho\|_{L^2}^2,
 \end{equation}

 \begin{equation}\label{es-varphi-k}
  \frac12\frac{d}{dt}\|\nabla^k \varphi\|_{L^2}^2 + \mu\|\nabla^{k+1} \varphi\|_{L^2}^2 + \nu\|\nabla^k\varphi\|_{L^2}^2 -\gamma\left\langle \nabla^k \varrho, \nabla^k \varphi\right\rangle=0,
 \end{equation}

 \begin{equation}\label{es-varrho-u-k}
  \frac{d}{dt}\langle\nabla^{k}\varrho, \nabla^{k-1}{\bf u} \rangle
  +\frac{P'(\bar\rho)}{2\bar\rho}\|\nabla^k \varrho\|_{L^2}^2
  \le \left(\frac{\alpha^2\bar\rho}{P'(\bar\rho)} +2\bar\rho\right)\|\nabla^{k-1}{\bf u}\|_1^2 +\frac{\beta^2\bar\rho}{P'(\bar\rho)}\|\nabla^{k}\varphi\|_{L^2}^2,
 \end{equation}
 and
 \begin{equation}\label{es-u-varphi-k}
 \begin{split}
  &\frac{d}{dt}\langle\nabla^{k-1}{\bf u}, \nabla^k\varphi\rangle +\frac{P'(\bar\rho)}{\bar\rho}\langle\nabla^k \varrho, \nabla^k\varphi\rangle +\mu\langle\nabla^k{\bf u}, \nabla^{k+1}\varphi\rangle +(\alpha +\nu)\langle\nabla^{k-1}{\bf u}, \nabla^k\varphi\rangle
  \\
  &\quad-\beta\|\nabla^k \varphi\|_{L^2}^2 -\gamma\langle\nabla^{k-1}{\bf u}, \nabla^k\varrho\rangle
  \\
 &\lesssim \mathcal{E}(t)\left(\|\nabla^k\varrho\|_{L^2}^2 +\|\nabla^k{\bf u}\|_{L^2}^2 +\|\nabla^k\varphi\|_{L^2}^2 \right).
 \end{split}
 \end{equation}
\end{lemma}

\begin{proof}
 First, we intend to derive the energy estimates \eqref{es-varrho-0}, \eqref{es-varrho-k-const} and \eqref{es-varrho-k} on $\varrho$ and its derivatives from the first equation of \eqref{2eq-var}. By taking the inner product $\langle \eqref{2eq-var}_1, \varrho\rangle$, and then using the H\"older inequality, we can arrive at
 \begin{equation}\label{es-varrho-0-1}
  \begin{split}
  \frac12\frac{d}{dt}\|\varrho\|_{L^2}^2 +\bar\rho\langle {\rm div}{\bf u}, \varrho\rangle
  &= -\langle {\bf u}\cdot \nabla\varrho + \varrho {\rm div}{\bf u}, \varrho\rangle
  \\
  &\le \|(\varrho, {\bf u})\|_{L^3}\|\nabla(\varrho, {\bf u})\|_{L^2}\|\varrho\|_{L^6}
  \\
  &\le C\mathcal{E}(t)\|\nabla(\varrho, {\bf u})\|_{L^2}^2.
  \end{split}
 \end{equation}
 Thus via integration by parts, we can easily obtain \eqref{es-varrho-0}.

Taking $\langle \nabla^k\eqref{2eq-var}_1, \nabla^k\varrho\rangle$ with $k=1,2$, we have
 \begin{equation}\label{es-varrho-k-const-1}
  \frac12\frac{d}{dt}\|\nabla^k\varrho\|_{L^2}^2 +\left\langle \nabla^k\left((\varrho +\bar\rho){\rm div}{\bf u}\right), \nabla^k\varrho \right\rangle
  = -\left\langle \nabla^k\left({\bf u}\cdot\nabla\varrho\right), \nabla^k\varrho \right\rangle.
 \end{equation}
 We estimates the terms in \eqref{es-varrho-k-const-1} as:
 \begin{align}\label{es-varrho-k-const-2}
 \begin{split}
  &\left\langle \nabla^k\left((\varrho +\bar\rho){\rm div}{\bf u}\right), \nabla^k\varrho \right\rangle
  \\
  &
  = \left\langle \left((\varrho +\bar\rho)\nabla^k{\rm div}{\bf u}\right), \nabla^k\varrho \right\rangle
    +\left\langle \left([\nabla^k, (\varrho +\bar\rho)]{\rm div}{\bf u}\right), \nabla^k\varrho \right\rangle
  \\
  &
  = -\bar\rho\left\langle \nabla^k{\bf u}, \nabla^{k+1}\varrho \right\rangle
     +\left\langle \varrho\nabla^k{\rm div}{\bf u}, \nabla^{k}\varrho \right\rangle
     +\left\langle \left([\nabla^k, (\varrho +\bar\rho)]{\rm div}{\bf u}\right), \nabla^k\varrho \right\rangle
  \\
  &
  \geq -\bar\rho\left\langle \nabla^k{\bf u}, \nabla^{k+1}\varrho \right\rangle
          -C\|\varrho\|_{L^\infty}\|\nabla^{k+1}{\bf u}\|_{L^2}\|\nabla^{k}\varrho\|_{L^2}
  \\
  &
   \quad -C\left(\|\nabla^k\varrho\|_{L^2}\|{\rm div}{\bf u}\|_{L^\infty} + \|\nabla \varrho\|_{L^3}\|\nabla^{k-1}{\rm div}{\bf u}\|_{L^6}\right)\|\nabla^k\varrho\|_{L^2}
  \\
  &
  \geq -\bar\rho\left\langle \nabla^k{\bf u}, \nabla^{k+1}\varrho \right\rangle
          -C\mathcal{E}(t)\left(\|\nabla^{k}\varrho\|_{L^2}^2 +\|\nabla^{k+1}{\bf u}\|_{L^2}^2\right),
  \end{split}
 \end{align}
 where the commutator $[\nabla^k, (\varrho +\bar\rho)]{\rm div}{\bf u}$ is defined in Lemma \ref{1commutator} and the estimate in Lemma \ref{1commutator} is used. Similarly, we can get
 \begin{equation}\label{es-varrho-k-const-3}
 \begin{split}
  -\left\langle \nabla^k\left({\bf u}\cdot\nabla\varrho\right), \nabla^k\varrho \right\rangle
  &
  = -\left\langle [\nabla^k, {\bf u}]\cdot\nabla\varrho, \nabla^k\varrho \right\rangle
     +\left\langle {\rm div}{\bf u}\nabla^k\varrho, \nabla^k\varrho \right\rangle
  \\
  &
  \le C\mathcal{E}(t)\left(\|\nabla^{k}\varrho\|_{L^2}^2 +\|\nabla^{k+1}{\bf u}\|_{L^2}^2\right),
  \end{split}
 \end{equation}
then combining \eqref{es-varrho-k-const-1} with \eqref{es-varrho-k-const-2}--\eqref{es-varrho-k-const-3} yields \eqref{es-varrho-k-const}.

Taking $\langle \frac{P'(\varrho +\bar\rho)}{(\varrho +\bar\rho)^2}\nabla^k\eqref{2eq-var}_1, \nabla^k\varrho\rangle$ with $1\le k\le 3$, we can obtain
 \begin{equation}\label{es-varrho-k-1}
  \left\langle \frac{P'(\varrho +\bar\rho)}{(\varrho +\bar\rho)^2}\nabla^k\varrho_t, \nabla^k\varrho \right\rangle
 +\left\langle \frac{P'(\varrho +\bar\rho)}{(\varrho +\bar\rho)^2}\nabla^k\left((\varrho +\bar\rho){\rm div}{\bf u}\right), \nabla^k\varrho \right\rangle
 = - \left\langle \frac{P'(\varrho +\bar\rho)}{(\varrho +\bar\rho)^2} \nabla^k({\bf u}\cdot\nabla\varrho),\nabla^k\varrho\right\rangle.
 \end{equation}
 By using the equation \eqref{2eq-var}$_1$, the H\"older inequality, Lemma \ref{1interpolation}--Lemma \ref{infty} and the a priori assumption \eqref{E-t-priori}, we have
 \begin{equation}\label{es-varrho-k-2}
 \begin{split}
  \left\langle \frac{P'(\varrho +\bar\rho)}{(\varrho +\bar\rho)^2}\nabla^k\varrho_t, \nabla^k\varrho \right\rangle
 &
 = \frac{d}{dt}\left\langle \frac{P'(\varrho +\bar\rho)}{2(\varrho +\bar\rho)^2}\nabla^k\varrho, \nabla^k\varrho \right\rangle
 -\left\langle \left(\frac{P'(\varrho +\bar\rho)}{2(\varrho +\bar\rho)^2}\right)'_\varrho \varrho_t\nabla^k\varrho, \nabla^k\varrho \right\rangle
 \\
 &
 \geq \frac{d}{dt}\left\langle \frac{P'(\varrho +\bar\rho)}{2(\varrho +\bar\rho)^2}\nabla^k\varrho, \nabla^k\varrho \right\rangle
 - C\|{\rm div}((\varrho +\bar\rho){\bf u})\|_{L^\infty}\|\nabla^k\varrho\|_{L^2}^2
 \\
 &
 \geq \frac{d}{dt}\left\langle \frac{P'(\varrho +\bar\rho)}{2(\varrho +\bar\rho)^2}\nabla^k\varrho, \nabla^k\varrho \right\rangle
 - C\mathcal{E}(t)\|\nabla^k\varrho\|_{L^2}^2,
 \end{split}
 \end{equation}
 \begin{align}\label{es-varrho-k-3}
 &\left\langle \frac{P'(\varrho +\bar\rho)}{(\varrho +\bar\rho)^2}\nabla^k\left((\varrho +\bar\rho){\rm div}{\bf u}\right), \nabla^k\varrho \right\rangle
 \notag\\
 &
 = \left\langle \frac{P'(\varrho +\bar\rho)}{\varrho +\bar\rho}\nabla^k{\rm div}{\bf u}, \nabla^k\varrho \right\rangle
 +\left\langle \frac{P'(\varrho +\bar\rho)}{(\varrho +\bar\rho)^2}\left[\nabla^k, \varrho +\bar\rho\right]{\rm div}{\bf u}, \nabla^k\varrho \right\rangle
  \notag\\
 &
 = - \left\langle \frac{P'(\varrho +\bar\rho)}{\varrho +\bar\rho}\nabla^k{\bf u}, \nabla^{k+1}\varrho \right\rangle -\left\langle \left(\frac{P'(\varrho +\bar\rho)}{\varrho +\bar\rho}\right)'_\varrho \nabla^k{\bf u} \cdot\nabla\varrho, \nabla^k\varrho \right\rangle
 \notag\\
 &
 \quad+\left\langle \frac{P'(\varrho +\bar\rho)}{(\varrho +\bar\rho)^2}\left[\nabla^k, \varrho +\bar\rho\right]{\rm div}{\bf u}, \nabla^k\varrho \right\rangle  \\
 &
 \geq - \left\langle \frac{P'(\varrho +\bar\rho)}{\varrho +\bar\rho}\nabla^k{\bf u}, \nabla^{k+1}\varrho \right\rangle
 - C\|\nabla(\varrho, {\bf u})\|_{L^\infty}\|\nabla^k(\varrho, {\bf u})\|_{L^2}\|\nabla^k\varrho\|_{L^2}
 \notag \\
 &
 \geq - \left\langle \frac{P'(\varrho +\bar\rho)}{\varrho +\bar\rho}\nabla^k{\bf u}, \nabla^{k+1}\varrho \right\rangle
 - C\mathcal{E}(t)\|\nabla^k(\varrho, {\bf u})\|_{L^2}^2,\notag
 \end{align}
 and
 \begin{equation}\label{es-varrho-k-4}
 \begin{split}
  &- \left\langle \frac{P'(\varrho +\bar\rho)}{(\varrho +\bar\rho)^2} \nabla^k({\bf u} \cdot\nabla\varrho),\nabla^k\varrho\right\rangle
  \\
  &
  = - \left\langle \frac{P'(\varrho +\bar\rho)}{(\varrho +\bar\rho)^2} [\nabla^k, {\bf u}] \cdot\nabla\varrho,\nabla^k\varrho\right\rangle +\left\langle {\rm div}\left(\frac{P'(\varrho +\bar\rho)}{(\varrho +\bar\rho)^2} {\bf u}\right),\frac12|\nabla^k\varrho|^2\right\rangle
  \\&\le C\mathcal{E}(t)\|\nabla^k(\varrho, {\bf u})\|_{L^2}^2.
 \end{split}
 \end{equation}
 Thus plugging \eqref{es-varrho-k-2}--\eqref{es-varrho-k-4} into \eqref{es-varrho-k-1} yields \eqref{es-varrho-k}.

 Now we turn to derive the energy estimates on ${\bf u}$ and its derivatives. Here we only prove  \eqref{es-u-k}, and \eqref{es-u-0} can be deduced in a easier way. Taking $\langle \nabla^k \eqref{2eq-var}_2, \nabla^k{\bf u}\rangle$ with $1\le k\le 3$, we can obtain
 \begin{equation}\label{es-u-k-1}
  \frac12\frac{d}{dt}\|\nabla^k{\bf u}\|_{L^2}^2 +\left\langle\nabla^k\left(\frac{P'(\varrho +\bar\rho)}{\varrho +\bar\rho}\nabla\varrho\right), \nabla^k{\bf u}\right\rangle +\alpha\|\nabla^k{\bf u}\|_{L^2}^2 -\beta\langle\nabla^{k+1}\varphi, \nabla^k{\bf u}\rangle = -\langle\nabla^k({\bf u}\cdot\nabla {\bf u}), \nabla^k{\bf u}\rangle.
  \end{equation}
 As in the proof of \eqref{es-varrho-k}, we have
 \begin{equation}\label{es-u-k-2}
 \begin{split}
  \left\langle\nabla^k\left(\frac{P'(\varrho +\bar\rho)}{\varrho +\bar\rho}\nabla\varrho\right), \nabla^k{\bf u}\right\rangle
  &= \left\langle\frac{P'(\varrho +\bar\rho)}{\varrho +\bar\rho}\nabla^{k+1}\varrho, \nabla^k{\bf u}\right\rangle  +\left\langle\left[\nabla^k, \frac{P'(\varrho +\bar\rho)}{\varrho +\bar\rho}\right]\nabla\varrho, \nabla^k{\bf u}\right\rangle
  \\
  &\geq \left\langle \frac{P'(\varrho +\bar\rho)}{\varrho +\bar\rho}\nabla^{k+1}\varrho, \nabla^k{\bf u}\right\rangle
  - C\mathcal{E}(t)\|\nabla^k(\varrho, {\bf u})\|_{L^2}^2,
  \end{split}
  \end{equation}
 and
 \begin{equation}\label{es-u-k-3}
 \begin{split}
  -\langle\nabla^k({\bf u}\cdot\nabla {\bf u}), \nabla^k{\bf u}\rangle
  & = -\langle[\nabla^k, u]\cdot\nabla {\bf u}, \nabla^k{\bf u}\rangle +\left\langle{\rm div}{\bf u}, \frac12|\nabla^k{\bf u}|^2\right\rangle
  \\
  &\le C\mathcal{E}(t)\|\nabla^k{\bf u}\|_{L^2}^2.
 \end{split}
 \end{equation}
 Hence by combining \eqref{es-u-k-1} with \eqref{es-u-k-2}--\eqref{es-u-k-3} and using the a priori assumption \eqref{E-t-priori}, we can deduce \eqref{es-u-k}.

 In order to get \eqref{es-varphi-0}, we take the inner product $\langle \nu\eqref{2eq-var}_3 - \gamma\eqref{2eq-var}_1, \nu\varphi -\gamma\varrho\rangle$ as
 \begin{equation}\label{es-varphi-0-1}
  \frac12\frac{d}{dt}\|\nu\varphi - \gamma\varrho\|_{L^2}^2 -\langle\mu\nu\Delta\varphi, \nu\varphi - \gamma\varrho\rangle +\nu\|\nu\varphi - \gamma\varrho\|_{L^2}^2
  -\langle\gamma(\varrho +\bar\rho){\rm div}{\bf u}, \nu\varphi - \gamma\varrho\rangle
  = \langle \gamma {\bf u}\cdot\nabla\varrho, \nu\varphi - \gamma\varrho\rangle.
  \end{equation}
 By using the Cauchy inequality and the a priori assumption \eqref{E-t-priori}, we have
 \begin{equation}\label{es-varphi-0-2}
 \begin{split}
  -\langle\mu\nu\Delta\varphi, \nu\varphi - \gamma\varrho\rangle
  &= \mu\nu^2\|\nabla\varphi\|_{L^2}^2 - \mu\nu\gamma\langle\nabla\varphi, \nabla\varrho\rangle
  \\
  &\geq \frac{\mu\nu^2}2\|\nabla\varphi\|_{L^2}^2 - 2\mu\gamma^2\|\nabla\varrho\|_{L^2}^2,
  \end{split}
  \end{equation}
 \begin{equation}\label{es-varphi-0-3}
 \begin{split}
  -\langle\gamma(\varrho +\bar\rho){\rm div}{\bf u}, \nu\varphi - \gamma\varrho\rangle
  &
  \geq -\frac{\gamma^2}\nu(\|\varrho\|_{L^\infty} +\bar\rho)^2\|{\rm div} u \|_{L^2}^2 -\frac\nu4\|\nu\varphi - \gamma\varrho\|_{L^2}^2
  \\
  &
  \geq -\frac{\gamma^2}\nu(\mathcal{E}(t) +\bar\rho)^2\|{\rm div} u \|_{L^2}^2 -\frac\nu4\|\nu\varphi - \gamma\varrho\|_{L^2}^2
  \\
  &\geq -\frac{5\gamma^2\bar\rho^2}{4\nu}\|{\rm div}{\bf u}\|_{L^2}^2 - \frac\nu4\|\nu\varphi - \gamma\varrho\|_{L^2}^2,
  \end{split}
 \end{equation}
 and
 \begin{equation}\label{es-varphi-0-4}
 \begin{split}
  \langle \gamma {\bf u}\cdot\nabla\varrho, \nu\varphi - \gamma\varrho\rangle
  &\le \frac{\gamma^2}{\nu}\|{\bf u}\|_{L^\infty}^2\|\nabla\varrho\|_{L^2}^2 + \frac{\nu}4\|\nu\varphi - \gamma\varrho\|_{L^2}^2
  \\
  &\le \frac{\gamma^2}{\nu}\mathcal{E}(t)^2\|\nabla\varrho\|_{L^2}^2 + \frac{\nu}4\|\nu\varphi - \gamma\varrho\|_{L^2}^2
  \\
  &\le \mu\gamma^2\|\nabla\varrho\|_{L^2}^2 + \frac{\nu}4\|\nu\varphi - \gamma\varrho\|_{L^2}^2,
  \end{split}
  \end{equation}
 thus plugging \eqref{es-varphi-0-2}--\eqref{es-varphi-0-4} into \eqref{es-varphi-0-1}, we can obtain \eqref{es-varphi-0}.

\eqref{es-varphi-k} can be easily deduced by taking $\langle \nabla^k\eqref{2eq-var}_3, \nabla^k\varphi\rangle$, we omit the detail. So it remains \eqref{es-varrho-u-k} and \eqref{es-u-varphi-k} to prove. Taking $\langle \nabla^k\eqref{2eq-var}_1, \nabla^{k-1}{\bf u}\rangle +\langle\nabla^k\varrho, \nabla^{k-1}\eqref{2eq-var}_2\rangle$ with $1\le k\le 3$ yields
\begin{equation}\label{es-varrho-u-k-1}
\begin{split}
 &\frac{d}{dt}\langle \nabla^k\varrho, \nabla^{k-1}{\bf u}\rangle +\langle\nabla^k\left((\varrho +\bar\rho){\rm div}{\bf u}\right), \nabla^{k-1}{\bf u}\rangle + \left\langle\nabla^k\varrho, \nabla^{k-1}\left(\frac{P'(\varrho +\bar\rho)}{\varrho +\bar\rho}\nabla\varrho\right) \right\rangle
 \\
 &
 = -\langle \nabla^k({\bf u}\cdot\nabla\varrho), \nabla^{k-1}{\bf u}\rangle
    - \langle \nabla^k\varrho, \nabla^{k-1}({\bf u}\cdot\nabla {\bf u})\rangle
    - \langle \nabla^k\varrho, \nabla^{k-1}(\alpha {\bf u} - \beta\nabla\varphi)\rangle.
 \end{split}
\end{equation}
Since by using \eqref{E-t-priori} and Lemmas \ref{1interpolation}--\ref{infty} again, we have
\begin{equation}\label{es-varrho-u-k-2}
\begin{split}
 \langle\nabla^k\left((\varrho +\bar\rho){\rm div}{\bf u}\right), \nabla^{k-1}{\bf u}\rangle
 &
 =\langle\nabla^{k-1}\left((\varrho +\bar\rho){\rm div}{\bf u}\right), \nabla^{k-1}{\rm div}{\bf u}\rangle
 \\
 &
 \geq -\frac{3\bar\rho}{2} \|\nabla^k{\bf u}\|_{L^2}^2 - C\mathcal{E}(t)\|\nabla^k\varrho\|_{L^2}^2,
 \end{split}
\end{equation}
\begin{equation}\label{es-varrho-u-k-3}
\begin{split}
 &\left\langle\nabla^k\varrho, \nabla^{k-1}\left(\frac{P'(\varrho +\bar\rho)}{\varrho +\bar\rho}\nabla\varrho\right) \right\rangle
 \\
 &
 =\left\langle\frac{P'(\varrho +\bar\rho)}{\varrho +\bar\rho}\nabla^k\varrho, \nabla^k\varrho \right\rangle
   + \left\langle\nabla^k\varrho, \left[\nabla^{k-1}, \frac{P'(\varrho +\bar\rho)}{\varrho +\bar\rho}\right] \nabla\varrho\right\rangle
 \\
 &
 \geq \frac{P'(\bar\rho)}{\bar\rho}\|\nabla^k\varrho\|_{L^2}^2 +\left\langle\left(\frac{P'(\varrho +\bar\rho)}{\varrho +\bar\rho} - \frac{P'(\bar\rho)}{\bar\rho}\right)\nabla^k\varrho, \nabla^k\varrho \right\rangle
   - C\mathcal{E}(t)\|\nabla^k\varrho\|_{L^2}^2
 \\
 &
 \geq \frac{P'(\bar\rho)}{\bar\rho}\|\nabla^k\varrho\|_{L^2}^2 - C\mathcal{E}(t)\|\nabla^k\varrho\|_{L^2}^2,
 \end{split}
 \end{equation}
 \begin{equation}\label{es-varrho-u-k-4}
 \begin{split}
  &-\langle \nabla^k({\bf u}\cdot\nabla\varrho), \nabla^{k-1}{\bf u}\rangle
    - \langle \nabla^k\varrho, \nabla^{k-1}({\bf u}\cdot\nabla {\bf u})\rangle
  \\
  &= -\langle \nabla^{k-1}({\bf u}\cdot\nabla\varrho), \nabla^{k-1}{\rm div}{\bf u}\rangle
    - \langle \nabla^k\varrho, \nabla^{k-1}({\bf u}\cdot\nabla {\bf u})\rangle
  \\
  &\le C\mathcal{E}(t)\|\nabla^k(\varrho, {\bf u})\|_{L^2}^2,
  \end{split}
 \end{equation}
 and
 \begin{equation}\label{es-varrho-u-k-5}
    - \langle \nabla^k\varrho, \nabla^{k-1}(\alpha {\bf u} - \beta\nabla\varphi)\rangle
    \le \frac{P'(\bar\rho)}{4\bar\rho}\|\nabla^k\varrho\|_{L^2}^2 +\frac{\alpha^2\bar\rho}{P'(\bar\rho)}\|\nabla^{k-1}{\bf u}\|_{L^2}^2 +\frac{\beta^2\bar\rho}{P'(\bar\rho)}\|\nabla^{k}\varphi\|_{L^2}^2.
 \end{equation}
 Thus by combining \eqref{es-varrho-u-k-1} with \eqref{es-varrho-u-k-2}--\eqref{es-varrho-u-k-5}, and using \eqref{E-t-priori} again, we can obtain \eqref{es-varrho-u-k}.

 Taking $\langle \nabla^{k-1}\eqref{2eq-var}_2, \nabla^{k}\varphi\rangle +\langle\nabla^{k-1}{\bf u}, \nabla^{k}\eqref{2eq-var}_3\rangle$ with $1\le k\le 3$ yields
\begin{equation}\label{es-u-varphi-k-1}
\begin{split}
 &\frac{d}{dt}\langle \nabla^{k-1}{\bf u}, \nabla^k\varphi\rangle + \left\langle \nabla^{k-1}\left(\frac{P'(\varrho +\bar\rho)}{\varrho +\bar\rho}\nabla\varrho\right), \nabla^k\varphi \right\rangle +(\alpha +\nu)\langle \nabla^{k-1}{\bf u}, \nabla^k\varphi\rangle -\beta\|\nabla^k\varphi\|_{L^2}^2
 \\
 & \quad+\mu\langle \nabla^{k}{\bf u}, \nabla^{k+1}\varphi\rangle -\gamma\langle\nabla^{k-1}{\bf u}, \nabla^k\varrho\rangle
 \\
 &= -\langle\nabla^{k-1}({\bf u}\cdot\nabla {\bf u}), \nabla^k\varphi\rangle.
 \end{split}
\end{equation}
This together with
\begin{equation}\nonumber
\begin{split}
 &\left\langle \nabla^{k-1}\left(\frac{P'(\varrho +\bar\rho)}{\varrho +\bar\rho}\nabla\varrho\right), \nabla^k\varphi \right\rangle
 \\
 &= \frac{P'(\bar\rho)}{\bar\rho} \langle \nabla^{k}\varrho, \nabla^k\varphi \rangle + \left\langle \left(\frac{P'(\varrho +\bar\rho)}{\varrho +\bar\rho} - \frac{P'(\bar\rho)}{\bar\rho}\right)\nabla^{k}\varrho, \nabla^k\varphi \right\rangle
 +\left\langle \left[\nabla^{k-1}, \frac{P'(\varrho +\bar\rho)}{\varrho +\bar\rho}\right]\nabla\varrho, \nabla^k\varphi \right\rangle
 \\
 &\geq \frac{P'(\bar\rho)}{\bar\rho} \langle \nabla^{k}\varrho, \nabla^k\varphi \rangle -C\mathcal{E}(t)\|\nabla^k(\varrho, \varphi)\|_{L^2}^2,
 \end{split}
\end{equation}
 and
 \begin{equation}\nonumber
  -\langle\nabla^{k-1}({\bf u}\cdot\nabla {\bf u}), \nabla^k\varphi\rangle
  \le \nabla^k(\varrho, {\bf u})\|_{L^2}^2
 \end{equation}
 yields \eqref{es-u-varphi-k}.
\end{proof}

\begin{lemma}\label{es-nonli-instab-en}
 Under the assumption of Lemma \ref{a priori-sol}, we have
 \begin{equation}\label{es-energy1}
   \|(\varrho, {\bf u}, \varphi)(t)\|_3^2 \le C\|(\varrho_0, {\bf u}_0, \varphi_0)\|_3^2 + C\int_0^t \|(\varrho, {\bf u}, \varphi)(\tau)\|_{L^2}^2d\tau.
 \end{equation}
\end{lemma}

\begin{proof}
 First, adding up \eqref{es-varrho-0}--\eqref{es-varphi-0} yields
 \begin{equation}\label{insta-E1-1}
 \begin{split}
  \frac{d}{dt}(\|\varrho\|_{L^2}^2 +\|{\bf u}\|_{L^2}^2 +\|\varphi\|_{L^2}^2)
  &\le C\|{\bf u}\|_{L^2}\|\nabla(\varrho, \varphi)\|_{L^2} +C\|\nabla(\varrho, {\bf u})\|_{L^2}^2
  \\
  &\le C\|{\bf u}\|_{L^2}^2 +C\|(\varrho, {\bf u}, \varphi)\|_{L^2}^\frac43\|\nabla^3(\varrho, {\bf u}, \varphi)\|_{L^2}^\frac23
  \\
  &\le C(\epsilon)\|(\varrho, {\bf u}, \varphi)\|_{L^2}^2 +\epsilon\|\nabla^3(\varrho, {\bf u}, \varphi)\|_{L^2}^2
  \end{split}
  \end{equation}
 with some sufficiently small positive constant $\epsilon$ to be determined.
 By taking the summation $\eqref{es-varrho-k} +\eqref{es-u-k} +\frac{\beta}{\mu}\times\eqref{es-u-varphi-k} +\frac{2\beta^2}{\mu\nu}\times\eqref{es-varphi-k} +K_1\times\eqref{es-varrho-u-k}$ with $k=3$ and $K_1 = \frac{\min\left\{\frac\alpha 4, \frac{\beta^2}{4\mu}\right\}}{\frac{\alpha^2\bar\rho}{P'(\bar\rho)} +2\bar\rho +\frac{\beta^2 \bar\rho}{P'(\bar\rho)}}$, one can arrive at
 \begin{equation}\nonumber
 \begin{split}
  &\frac{d}{dt}\left(\left\langle\frac{P'(\varrho+\bar\rho)}{2(\varrho +\bar\rho)^2}\nabla^3\varrho, \nabla^3\varrho\right\rangle +\frac12\|\nabla^3 {\bf u}\|_{L^2}^2 +\frac\beta\mu\langle\nabla^2 {\bf u}, \nabla^3\varphi\rangle +\frac{\beta^2}{\mu\nu}\|\nabla^3\varphi\|_{L^2}^2 +K_1\langle\nabla^3\varrho, \nabla^2 {\bf u}\rangle\right)
  \\
  & \quad+\frac\alpha2\|\nabla^3{\bf u}\|_{L^2}^2 +\frac{\beta^2}{\nu}\|\nabla^4\varphi\|_{L^2}^2 +\frac{\beta^2}{2\mu}\|\nabla^3\varphi\|_{L^2}^2 +\frac{K_1P'(\bar\rho)}{2\bar\rho}\|\nabla^3\varrho\|_{L^2}^2
  \\
  &\le C\mathcal{E}(t)\|\nabla^3(\varrho, {\bf u}, \varphi)\|_{L^2}^2 + C\|\nabla^2(\varrho, {\bf u}, \varphi)\|_{L^2}^2,
  \end{split}
 \end{equation}
 where $\mathcal{E}(t)$ is defined in \eqref{E-t-priori} and we used $\langle \nabla^3\varphi, \nabla^3\varrho\rangle = -\langle \nabla^4\varphi, \nabla^2\varrho\rangle \le \|\nabla^4\varphi\|_{L^2}\|\nabla^2\varrho\|_{L^2}$. Then by using the Sobolev inequality $\|\nabla^2f\|_{L^2}\le \|f\|_{L^2}^\frac13\|\nabla^3f\|_{L^2}^\frac23$, the Young inequality and the a priori assumption \eqref{E-t-priori}, we can get
 \begin{equation}\label{insta-E1-2}
 \begin{split}
  &\frac{d}{dt}\left(\left\langle\frac{P'(\varrho+\bar\rho)}{2(\varrho +\bar\rho)^2}\nabla^3\varrho, \nabla^3\varrho\right\rangle +\frac12\|\nabla^3 {\bf u}\|_{L^2}^2 +\frac\beta\mu\langle\nabla^2 {\bf u}, \nabla^3\varphi\rangle +\frac{\beta^2}{\mu\nu}\|\nabla^3\varphi\|_{L^2}^2 +K_1\langle\nabla^3\varrho, \nabla^2 {\bf u}\rangle\right)
  \\
  &\quad +\frac\alpha4\|\nabla^3{\bf u}\|_{L^2}^2 +\frac{\beta^2}{2\nu}\|\nabla^4\varphi\|_{L^2}^2 +\frac{\beta^2}{4\mu}\|\nabla^3\varphi\|_{L^2}^2 +\frac{K_1P'(\bar\rho)}{4\bar\rho}\|\nabla^3\varrho\|_{L^2}^2
  \\
  &\le C\|(\varrho, {\bf u}, \varphi)\|_{L^2}^2.
  \end{split}
 \end{equation}
  Now, we take the summation $\eqref{insta-E1-2} +K_2\times \eqref{insta-E1-1}$ with $K_2$ being a  positive constant to be determined, then one can obtain
 \begin{equation}\label{insta-E1-3}
 \begin{split}
  &\frac{d}{dt}\left(\left\langle\frac{P'(\varrho+\bar\rho)}{2(\varrho +\bar\rho)^2}\nabla^3\varrho, \nabla^3\varrho\right\rangle +\frac12\|\nabla^3 {\bf u}\|_{L^2}^2 +\frac\beta\mu\langle\nabla^2 {\bf u}, \nabla^3\varphi\rangle +\frac{\beta^2}{\mu\nu}\|\nabla^3\varphi\|_{L^2}^2 +K_1\langle\nabla^3\varrho, \nabla^2 {\bf u}\rangle\right.
  \\
  &\qquad + K_2\left(\|\varrho\|_{L^2}^2 +\|{\bf u}\|_{L^2}^2 +\|\varphi\|_{L^2}^2\right)\bigg)
  \\
  & \quad+\frac\alpha4\|\nabla^3 {\bf u}\|_{L^2}^2 +\frac{\beta^2}{2\nu}\|\nabla^4\varphi\|_{L^2}^2 +\frac{\beta^2}{4\mu}\|\nabla^3\varphi\|_{L^2}^2 +\frac{K_1P'(\bar\rho)}{4\bar\rho}\|\nabla^3\varrho\|_{L^2}^2
  \\
  &\le C\|(\varrho, {\bf u}, \varphi)\|_{L^2}^2 + K_2 \epsilon\|\nabla^3(\varrho, {\bf u}, \varphi)\|_{L^2}^2.
  \end{split}
 \end{equation}
 Define
 \begin{align*}
 \mathcal{L}_1(t) &= \left\langle\frac{P'(\varrho+\bar\rho)}{2(\varrho +\bar\rho)^2}\nabla^3\varrho, \nabla^3\varrho\right\rangle +\frac12\|\nabla^3 {\bf u}\|_{L^2}^2 +\frac\beta\mu\langle\nabla^2 {\bf u}, \nabla^3\varphi\rangle +\frac{\beta^2}{\mu\nu}\|\nabla^3\varphi\|_{L^2}^2 +K_1\langle\nabla^3\varrho, \nabla^2 {\bf u}\rangle
  \\
  &\quad + K_2\left(\|\varrho\|_{L^2}^2 +\|{\bf u}\|_{L^2}^2 +\|\varphi\|_{L^2}^2\right).
 \end{align*}
 By using the Young inequality and taking $K_2$ appropriately large, for $\mathcal{E}(t)$ defined in \eqref{E-t-priori}, we have
 \begin{equation}\label{L2-E}
  \mathcal{L}_1(t) \approx \mathcal{E}(t)^2.
 \end{equation}

 Hence by letting $\epsilon$ sufficiently small, one can derive from \eqref{insta-E1-3} that
 \begin{equation}\label{insta-E1-4}
  \frac{d}{dt} \mathcal{L}_1(t) +\frac\alpha8\|\nabla^3 {\bf u}\|_{L^2}^2 +\frac{\beta^2}{2\nu}\|\nabla^4\varphi\|_{L^2}^2 +\frac{\beta^2}{8\mu}\|\nabla^3\varphi\|_{L^2}^2 +\frac{K_1P'(\bar\rho)}{8\bar\rho}\|\nabla^3\varrho\|_{L^2}^2
  \le C\|(\varrho, {\bf u}, \varphi)\|_{L^2}^2.
 \end{equation}
 Thanks to \eqref{L2-E}, integrating \eqref{insta-E1-4} on $[0, t]$ gives rise to \eqref{es-energy1}.
\end{proof}

\begin{lemma}\label{es-u-varphi-ener}
 Under the assumption of Lemma \ref{a priori-sol} and $\frac{\nu P'(\bar\rho)}{\gamma \bar\rho} -\beta >0$, we have
\begin{equation}\label{es-u-varphi-ener1}
 \begin{split}
  &\frac{d}{dt}\sum_{k=1}^3\left(\frac12\|\nabla^k{\bf u}\|_{L^2}^2 +\frac\beta\mu \langle\nabla^{k-1}{\bf u}, \nabla^k\varphi\rangle +\frac{\beta P'(\bar\rho)}{2\mu\gamma\bar\rho}\|\nabla^k\varphi\|_{L^2}^2 +\frac{\alpha +\nu}{2\mu}\|\nabla^{k-1}{\bf u}\|_{L^2}^2 +\frac{P'(\bar\rho)}{2\bar\rho^2}\|\nabla^k\varrho\|_{L^2}^2 \right.
  \\
  &\quad\qquad\left.+\frac{1}{2\bar\rho}\left( \frac{\gamma}{\mu}\left(\frac{\nu P'(\bar\rho)}{\gamma \bar\rho} -\beta\right) +\frac{\alpha P'(\bar\rho)}{\mu\bar\rho}\right)\|\nabla^{k-1}\varrho\|_{L^2}^2\right)
  \\
  &+\sum_{k=1}^3\left(\frac{\alpha}{2}\|\nabla^k{\bf u}\|_{L^2}^2 +\frac{\alpha(\alpha+\nu)}{2\mu}\| \nabla^{k-1}{\bf u}\|_{L^2}^2 +\frac{\beta P'(\bar\rho)}{\gamma\bar\rho}\|\nabla^{k+1} \varphi\|_{L^2}^2 +\frac{\beta}{2\mu}\left(\frac{\nu P'(\bar\rho)}{\gamma \bar\rho} -\beta\right)\|\nabla^k\varphi\|_{L^2}^2 \right)
  \\
  &\le
  \mathcal{E}(t)\|\nabla\varrho\|_2^2.
  \end{split}
  \end{equation}
\end{lemma}
\begin{proof}
 We take the summation $\eqref{es-u-k} +\frac{\beta}{\mu}\times \eqref{es-u-varphi-k} +\frac{\beta P'(\bar\rho)}{\mu\gamma\bar\rho} \times\eqref{es-varphi-k} +\frac{\alpha+\nu}{\mu}\times \eqref{es-u-0}$ while $k=1$ and the summation $\eqref{es-u-k}$ with $k +\frac{\beta}{\mu}\times \eqref{es-u-varphi-k} +\frac{\beta P'(\bar\rho)}{\mu\gamma\bar\rho} \times\eqref{es-varphi-k} +\frac{\alpha+\nu}{\mu}\times ``\eqref{es-u-k}$ with $k-1"$ while $k=2,3$ respectively, and then we can arrive at the following estimates:
 \begin{equation}\label{es-u-varphi-ener2}
 \begin{split}
  &\frac{d}{dt}\left(\frac12\|\nabla^k{\bf u}\|_{L^2}^2 +\frac\beta\mu \langle\nabla^{k-1}{\bf u}, \nabla^k\varphi\rangle +\frac{\beta P'(\bar\rho)}{2\mu\gamma\bar\rho}\|\nabla^k\varphi\|_{L^2}^2 +\frac{\alpha +\nu}{2\mu}\|\nabla^{k-1}{\bf u}\|_{L^2}^2\right)
  \\
  &\quad+\frac{5\alpha}{8}\|\nabla^k{\bf u}\|_{L^2}^2 +\frac{\alpha(\alpha+\nu)}{2\mu}\| \nabla^{k-1}{\bf u}\|_{L^2}^2 +\frac{\beta P'(\bar\rho)}{\gamma\bar\rho}\|\nabla^{k+1} \varphi\|_{L^2}^2 +\frac{\beta}{2\mu}\left(\frac{\nu P'(\bar\rho)}{\gamma \bar\rho} -\beta\right)\|\nabla^k\varphi\|_{L^2}^2
  \\
  &\quad+\left\langle\frac{P'(\varrho +\bar\rho)}{\varrho +\bar\rho}\nabla^{k+1}\varrho, \nabla^k{\bf u}\right\rangle
  +\left(\frac{(\alpha +\nu)P'(\bar\rho)}{\mu\bar\rho}
  -\frac{\beta\gamma}{\mu} \right)\langle\nabla^k\varrho, \nabla^{k-1}{\bf u}\rangle
  \\
  &\lesssim
 \left\{\begin{array}{lll}
  \mathcal{E}(t)\|\nabla\varrho\|_{L^2}^2,\quad while\quad k=1,
  \\
  \mathcal{E}(t)\|\nabla^{k-1}\varrho\|_1^2,\quad while \quad k=2,3.
 \end{array}\right.
  \end{split}
  \end{equation}
  Here we used the following estimate:
  \begin{equation}\nonumber
  \begin{split}
   \left\langle \frac{P'(\varrho +\bar\rho)}{\varrho +\bar\rho} \nabla^k\varrho, \nabla^{k-1}{\bf u}\right\rangle
   &=\left\langle \frac{P'(\bar\rho)}{\bar\rho} \nabla^k\varrho, \nabla^{k-1}{\bf u}\right\rangle
    +\left\langle \left(\frac{P'(\varrho +\bar\rho)}{\varrho +\bar\rho} - \frac{P'(\bar\rho)}{\bar\rho}\right) \nabla^k\varrho, \nabla^{k-1}{\bf u}\right\rangle
    \\
    &\geq \left\langle \frac{P'(\bar\rho)}{\bar\rho} \nabla^k\varrho, \nabla^{k-1}{\bf u}\right\rangle
     -C\mathcal{E}(t)\left(\|\nabla^k\varrho\|_{L^2}^2 +\|\nabla^{k-1}{\bf u}\|_{L^2}^2\right)
    \end{split}
   \end{equation}
   and the a priori assumption \eqref{E-t-priori}.
  In order to cancel out the last two terms on the left--hand side of \eqref{es-u-varphi-ener2}, we plus \eqref{es-u-varphi-ener2} with  $\eqref{es-varrho-k} + \frac{1}{\bar\rho}\left(\frac{(\alpha +\nu)P'(\bar\rho)}{\mu\bar\rho}
  -\frac{\beta\gamma}{\mu} \right)\times\eqref{es-varrho-0}$  while $k=1$ and $\eqref{es-varrho-k} + \frac{1}{\bar\rho}\left(\frac{(\alpha +\nu)P'(\bar\rho)}{\mu\bar\rho}
  -\frac{\beta\gamma}{\mu} \right) \times ``\eqref{es-varrho-k-const}$ with $k-1" $ while $k=2,3$ respectively,
 then by the a priori assumption \eqref{E-t-priori}, we can deduce
 \begin{equation}\label{es-u-varphi-ener4}
 \begin{split}
  &\frac{d}{dt}\left(\frac12\|\nabla^k{\bf u}\|_{L^2}^2 +\frac\beta\mu \langle\nabla^{k-1}{\bf u}, \nabla^k\varphi\rangle +\frac{\beta P'(\bar\rho)}{2\mu\gamma\bar\rho}\|\nabla^k\varphi\|_{L^2}^2 +\frac{\alpha +\nu}{2\mu}\|\nabla^{k-1}{\bf u}\|_{L^2}^2  \right.
  \\
  &\qquad\left. + \left\langle\frac{P'(\varrho +\bar\rho)}{2(\varrho +\bar\rho)^2}\nabla^k\varrho, \nabla^k\varrho\right\rangle
  +\frac{1}{2\bar\rho}\left(\frac{(\alpha +\nu)P'(\bar\rho)}{\mu\bar\rho}
  -\frac{\beta\gamma}{\mu} \right)\|\nabla^{k-1}\varrho\|_{L^2}^2\right)
  \\
  &\quad+\frac{\alpha}{2}\|\nabla^k{\bf u}\|_{L^2}^2 +\frac{\alpha(\alpha+\nu)}{2\mu}\| \nabla^{k-1}{\bf u}\|_{L^2}^2 +\frac{\beta P'(\bar\rho)}{\gamma\bar\rho}\|\nabla^{k+1} \varphi\|_{L^2}^2 +\frac{\beta}{2\mu}\left(\frac{\nu P'(\bar\rho)}{\gamma \bar\rho} -\beta\right)\|\nabla^k\varphi\|_{L^2}^2
  \\
  &\lesssim
 \left\{\begin{array}{lll}
  \mathcal{E}(t)\|\nabla\varrho\|_{L^2}^2,\quad while\quad k=1,
  \\
  \mathcal{E}(t)\|\nabla^{k-1}\varrho\|_1^2,\quad while \quad k=2,3.
 \end{array}\right.
  \end{split}
  \end{equation}
 Hence by adding up \eqref{es-u-varphi-ener4} from $k=1$ to $3$, we can obtain \eqref{es-u-varphi-ener1}.
\end{proof}

 By taking the summation $\frac{P'(\bar\rho)}{16\mu\gamma^2\bar\rho}\times\eqref{es-varphi-0} + \eqref{es-varrho-u-k}$ from $k=1$ to $3$, we can also deduce the following estimate directly:
\begin{lemma}\label{es-dispper-varphi-ener}
 Under the assumption of Lemma \ref{a priori-sol} and $\frac{\nu P'(\bar\rho)}{\gamma \bar\rho} -\beta >0$, we have
 \begin{equation}\label{es-dispper-varphi-ener1}
 \begin{split}
  &\frac{d}{dt}\left(\frac{P'(\bar\rho)}{16\mu\gamma^2\bar\rho}\|\nu\varphi - \gamma\varrho\|_{L^2}^2 + \sum_{k=1}^3\langle\nabla^k\varrho, \nabla^{k-1}{\bf u}\rangle\right)
  +\frac{\nu P'(\bar\rho)}{16\mu\gamma^2\bar\rho}\|\nu\varphi - \gamma\varrho\|_{L^2}^2 +\frac{P'(\bar\rho)}{8\bar\rho}\|\nabla\varrho\|_2^2
  \\
  &\le \left(\frac{\bar\rho P'(\bar\rho)}{4\mu\nu} +\frac{2\alpha^2\bar\rho}{P'(\bar\rho)} +4\bar\rho\right)\|{\bf u}\|_3^2 +\frac{\beta^2\bar\rho}{P'(\bar\rho)}\|\nabla\varphi\|_{L^2}^2.
  \end{split}
  \end{equation}
\end{lemma}

Finally, in order to improve the optimal decay rate of the highest--order derivatives of the solutions, we have to estimate the high--frequency part of the solution. To this end, we apply the operator $\mathcal{F}^{-1}[(1-\psi(\xi))\mathcal{F}(\cdot)]$ to the system \eqref{2eq-var}, which gives rise to
\begin{equation}\label{2eq-var-H}
 \left\{\begin{array}{lll}
 \varrho^H_t + ((\varrho + \bar\rho){\rm div}{\bf u})^H = - ({\bf u}\cdot \nabla\varrho)^H,
 \\
 {\bf u}^H_t + \left(\frac{P'(\varrho + \bar\rho)}{\varrho + \bar\rho} \nabla\varrho\right)^H +\alpha {\bf u}^H - \beta \nabla \varphi^H = -({\bf u}\cdot \nabla {\bf u})^H,
 \\
 \varphi^H_t - \mu \Delta \varphi^H +\nu \varphi^H - \gamma \varrho^H = 0,
 \\
 (\varrho^H, {\bf u}^H, \varphi^H)(x,0) =(\varrho_{0}^H, {\bf u}_0^H, \varphi_0 )(x).
 \end{array}\right.
\end{equation}

\begin{lemma}\label{a priori-sol-H}
Under the assumption of Lemma \ref{a priori-sol}, we have the following estimates:
 \begin{equation}\label{es-varrho-H-3}
 \begin{split}
  &\frac{d}{dt}\left\langle \frac {P'(\varrho +\bar\rho)}{2(\varrho +\bar\rho)^2}\nabla^3\varrho^H, \nabla^3\varrho^H\right\rangle - \left\langle  \frac{P'(\varrho +\bar\rho)} {\varrho +\bar\rho}\nabla^3 {\bf u}^H, \nabla^4\varrho^H\right\rangle
  \\
  &
  \quad+ \frac{P'(\bar\rho)}{\bar\rho}\langle\nabla^3 {\bf u}^L, \nabla^4\varrho\rangle
   + \frac{P'(\bar\rho)}{\bar\rho}\langle\nabla^3 {\bf u}^H, \nabla^4\varrho^L\rangle
   \\
   &
 \le C(\|(\varrho, {\bf u}, \nabla\varrho, \nabla {\bf u})\|_{L^\infty} +\|\nabla\varrho\|_{L^3})\|\nabla^3(\varrho, {\bf u})\|_{L^2}^2,
 \end{split}
 \end{equation}

  \begin{equation}\label{es-u-H-3}
  \begin{split}
  &\frac12\frac{d}{dt}\|\nabla^3 {\bf u}^H\|_{L^2}^2 + \frac{3\alpha}4\|\nabla^3 {\bf u}^H\|_{L^2}^2 + \left\langle \frac{P'(\varrho +\bar\rho)} {\varrho +\bar\rho}\nabla^4 \varrho^H, \nabla^3 {\bf u}^H\right\rangle
   \\
   &
  \quad-\frac{P'(\bar\rho)}{\bar\rho}\langle\nabla^4\varrho^L, \nabla^3 {\bf u}\rangle
   -\frac{P'(\bar\rho)}{\bar\rho}\langle\nabla^4\varrho^H, \nabla^3 {\bf u}^L\rangle
   - \beta\langle\nabla^4\varphi^H, \nabla^3 {\bf u}^H\rangle
 \\
 &\le C\|\nabla \varrho\|_{L^\infty}\|\nabla^3\varrho\|_{L^2}^2,
 \end{split}
 \end{equation}

 \begin{equation}\label{es-varphi-H-3}
  \frac12\frac{d}{dt}\|\nabla^3 \varphi^H\|_{L^2}^2 + \mu\|\nabla^{4} \varphi^H\|_{L^2}^2 + \nu\|\nabla^3\varphi^H\|_{L^2}^2 -\gamma\left\langle \nabla^3 \varrho^H, \nabla^3 \varphi^H\right\rangle=0,
 \end{equation}

 \begin{equation}\label{es-varrho-u-H-3}
 \begin{split}
 & \frac{d}{dt}\langle\nabla^{3}\varrho^H, \nabla^{2}{\bf u}^H \rangle
  +\frac{P'(\bar\rho)}{2\bar\rho}\|\nabla^3 \varrho^H\|_{L^2}^2
  \\
  &
  \le \frac{\alpha^2\bar\rho}{P'(\bar\rho)} \|\nabla^{2}{\bf u}^H\|_{L^2}^2
  +\bar\rho\|\nabla^2{\rm div}{\bf u}^H\|_{L^2}^2
  +\frac{\beta^2\bar\rho}{P'(\bar\rho)}\|\nabla^{3}\varphi^H\|_{L^2}^2
  \\
  &
  \quad+C(\|\varrho\|_{L^\infty} +\|\nabla(\varrho, {\bf u})\|_{L^3})\|\nabla^3(\varrho, {\bf u})\|_{L^2}^2,
  \end{split}
 \end{equation}
 and
 \begin{align}\label{es-u-varphi-H-3}
  &\frac{d}{dt}\langle\nabla^{2}{\bf u}^H, \nabla^3\varphi^H\rangle +\frac{P'(\bar\rho)}{\bar\rho}\langle\nabla^3 \varrho^H, \nabla^3\varphi^H\rangle +\mu\langle\nabla^{4}\varphi^H, \nabla^3 {\bf u}^H\rangle
  \notag\\
 &\le\frac12\left(\frac{\nu P'(\bar\rho)}{\gamma\bar\rho} +\beta\right)\|\nabla^3\varphi^H\|_{L^2}^2 +\frac{K_1P'(\bar\rho)}{4\bar\rho}\|\nabla^3\varrho^H\|_{L^2}^2 +C\|\nabla^2 {\bf u}^H\|_{L^2}^2
 \\
 &
 \quad +C\left(\|\varrho\|_{L^\infty} +\|\nabla\varrho\|_{L^3}\right)\|\nabla^3(\varrho, \varphi)\|_{L^2}^2.\notag
 \end{align}
\end{lemma}
\begin{proof}
Taking $\langle \frac{P'(\varrho +\bar\rho)}{(\varrho +\bar\rho)^2}\nabla^3\eqref{2eq-var-H}_1, \nabla^3\varrho^H\rangle$, we can obtain
 \begin{equation}\label{es-varrho-H-3-1}
 \begin{split}
  &\left\langle \frac{P'(\varrho +\bar\rho)}{(\varrho +\bar\rho)^2}\nabla^3\varrho^H_t, \nabla^3\varrho^H \right\rangle
  +\left\langle \frac{P'(\varrho +\bar\rho)}{(\varrho +\bar\rho)^2}\nabla^3\left((\varrho +\bar\rho){\rm div}{\bf u}\right)^H, \nabla^3\varrho^H \right\rangle
  \\
  &= - \left\langle \frac{P'(\varrho +\bar\rho)}{(\varrho +\bar\rho)^2} \nabla^3({\bf u} \cdot\nabla\varrho)^H,\nabla^3\varrho^H\right\rangle.
 \end{split}
 \end{equation}
 By using  the H\"older inequality, Lemma \ref{1interpolation}--Lemma \ref{infty}, the a priori assumption \eqref{E-t-priori}, integration by parts and the fact that $\|\nabla f^L\|_{L^2}\lesssim \|f^L\|_{L^2}$ and $\|f\|_{L^2}\approx \|f^L\|_{L^2} +\|f^H\|_{L^2}$, we have
 \begin{equation}\label{es-varrho-H-3-2}
 \begin{split}
  \left\langle \frac{P'(\varrho +\bar\rho)}{(\varrho +\bar\rho)^2}\nabla^3\varrho^H_t, \nabla^3\varrho^H \right\rangle
 &
 = \frac{d}{dt}\left\langle \frac{P'(\varrho +\bar\rho)}{2(\varrho +\bar\rho)^2}\nabla^3\varrho^H, \nabla^3\varrho^H \right\rangle
 -\left\langle \left(\frac{P'(\varrho +\bar\rho)}{2(\varrho +\bar\rho)^2}\right)'_\varrho \varrho_t\nabla^3\varrho^H, \nabla^3\varrho^H \right\rangle
 \\
 &
 \geq \frac{d}{dt}\left\langle \frac{P'(\varrho +\bar\rho)}{2(\varrho +\bar\rho)^2}\nabla^3\varrho^H, \nabla^3\varrho^H \right\rangle
 - C\|{\rm div}((\varrho +\bar\rho)u)\|_{L^\infty}\|\nabla^3\varrho^H\|_{L^2}^2
 \\
 &
 \geq \frac{d}{dt}\left\langle \frac{P'(\varrho +\bar\rho)}{2(\varrho +\bar\rho)^2}\nabla^3\varrho^H, \nabla^3\varrho^H \right\rangle
 - C\|\nabla(\varrho, {\bf u})\|_{L^\infty}\|\nabla^3\varrho\|_{L^2}^2,
 \end{split}
 \end{equation}
 \begin{align}
 &\left\langle \frac{P'(\varrho +\bar\rho)}{(\varrho +\bar\rho)^2}\nabla^3\left((\varrho +\bar\rho){\rm div}{\bf u}\right)^H, \nabla^3\varrho^H \right\rangle
 \nonumber
 \\
 &
 = \left\langle \frac{P'(\varrho +\bar\rho)}{(\varrho +\bar\rho)^2}\nabla^3\left((\varrho +\bar\rho){\rm div}{\bf u}\right), \nabla^3\varrho \right\rangle
    - \left\langle \frac{P'(\varrho +\bar\rho)}{(\varrho +\bar\rho)^2}\nabla^3\left((\varrho +\bar\rho){\rm div}{\bf u}\right)^L, \nabla^3\varrho \right\rangle
  \nonumber
  \\
  &
    \quad- \left\langle \frac{P'(\varrho +\bar\rho)}{(\varrho +\bar\rho)^2}\nabla^3\left((\varrho +\bar\rho){\rm div}{\bf u}\right)^H, \nabla^3\varrho^L \right\rangle
  \nonumber
  \\
  &
  = - \left\langle \frac{P'(\varrho +\bar\rho)}{\varrho +\bar\rho}\nabla^3 {\bf u}, \nabla^{4}\varrho \right\rangle
 -\left\langle \left(\frac{P'(\varrho +\bar\rho)}{\varrho +\bar\rho}\right)'_\varrho \nabla^3 {\bf u} \cdot\nabla\varrho, \nabla^3\varrho \right\rangle
 \nonumber
 \\
 &
 \quad
  +\left\langle \frac{P'(\varrho +\bar\rho)}{(\varrho +\bar\rho)^2}\left[\nabla^3, \varrho +\bar\rho\right]{\rm div}{\bf u}, \nabla^3\varrho \right\rangle
 - \frac{P'(\bar\rho)}{\bar\rho}\langle\nabla^3{\rm div}{\bf u}^L, \nabla^3\varrho\rangle
 - \frac{P'(\bar\rho)}{\bar\rho^2}\langle\nabla^3(\varrho{\rm div}{\bf u})^L, \nabla^3\varrho\rangle
 \nonumber
 \\
 &
 \quad
 -\left\langle\left(\frac{P'(\varrho +\bar\rho)}{(\varrho +\bar\rho)^2} - \frac{P'(\bar\rho)}{\bar\rho^2}\right)\nabla^3((\varrho +\bar\rho){\rm div}{\bf u})^L, \nabla^3\varrho\right\rangle
 - \frac{P'(\bar\rho)}{\bar\rho}\langle\nabla^3{\rm div}{\bf u}^H, \nabla^3\varrho^L\rangle
 \nonumber
 \\
 &
 \quad
 - \frac{P'(\bar\rho)}{\bar\rho^2}\langle\nabla^3(\varrho{\rm div}{\bf u})^H, \nabla^3\varrho^L\rangle
-\left\langle\left(\frac{P'(\varrho +\bar\rho)}{(\varrho +\bar\rho)^2} - \frac{P'(\bar\rho)}{\bar\rho^2}\right)\nabla^3((\varrho +\bar\rho){\rm div}{\bf u})^H, \nabla^3\varrho^L\right\rangle   \nonumber
   \\
  &
  \geq - \left\langle \frac{P'(\varrho +\bar\rho)}{\varrho +\bar\rho}\nabla^3 {\bf u}, \nabla^{4}\varrho \right\rangle
 -\frac{P'(\bar\rho)}{\bar\rho}\langle\nabla^3 {\rm div}{\bf u}^L, \nabla^3\varrho\rangle
 - \frac{P'(\bar\rho)}{\bar\rho}\langle\nabla^3{\rm div} {\bf u}^H, \nabla^3\varrho^L\rangle
 \nonumber
 \\
 &
 \quad
 - C\|\nabla(\varrho, {\bf u})\|_{L^\infty}\|\nabla^3(\varrho, {\bf u})\|_{L^2}^2
 \nonumber
 \\
 &
 \quad - C\left(\left\|\frac{P'(\varrho +\bar\rho)}{(\varrho +\bar\rho)^2} - \frac{P'(\bar\rho)}{\bar\rho^2}\right\|_{L^\infty}\|\nabla^3((\varrho +\bar\rho){\rm div}{\bf u})^L\|_{L^2} + \|\nabla^3(\varrho{\rm div}{\bf u})^L\|_{L^2}\right)\|\nabla^3\varrho\|_{L^2}
 \nonumber
 \\
 &
 \quad - C\left(\left\|\frac{P'(\varrho +\bar\rho)}{(\varrho +\bar\rho)^2} - \frac{P'(\bar\rho)}{\bar\rho^2}\right\|_{L^\infty}\|\nabla^2((\varrho +\bar\rho){\rm div}{\bf u})^H\|_{L^2} + \|\nabla^2(\varrho{\rm div}{\bf u})^H\|_{L^2}\right)\|\nabla^4\varrho^L\|_{L^2}
   \nonumber
   \\
  &
  \geq - \left\langle \frac{P'(\varrho +\bar\rho)}{\varrho +\bar\rho}\nabla^3 {\bf u}, \nabla^{4}\varrho \right\rangle
 + \frac{P'(\bar\rho)}{\bar\rho}\langle\nabla^3 {\bf u}^L, \nabla^4\varrho\rangle
 + \frac{P'(\bar\rho)}{\bar\rho}\langle\nabla^3 {\bf u}^H, \nabla^4\varrho^L\rangle
 \nonumber
 \\
 &
 \quad
 - C\|\nabla(\varrho, {\bf u})\|_{L^\infty}\|\nabla^3(\varrho, {\bf u})\|_{L^2}^2
 - C\left(\|\varrho\|_{L^\infty}\|\nabla^2((\varrho +\bar\rho){\rm div}{\bf u})\|_{L^2} + \|\nabla^2(\varrho{\rm div}{\bf u})\|_{L^2}\right)\|\nabla^3\varrho\|_{L^2}
  \nonumber
   \\
  &
  \geq - \left\langle \frac{P'(\varrho +\bar\rho)}{\varrho +\bar\rho}\nabla^3 {\bf u}, \nabla^{4}\varrho \right\rangle
 + \frac{P'(\bar\rho)}{\bar\rho}\langle\nabla^3 {\bf u}^L, \nabla^4\varrho\rangle
 + \frac{P'(\bar\rho)}{\bar\rho}\langle\nabla^3 {\bf u}^H, \nabla^4\varrho^L\rangle
 \nonumber
 \\
 &
 \quad
 - C(\|\varrho\|_{L^\infty} + \|\nabla{\bf u}\|_{L^3} + \|\nabla(\varrho, {\bf u})\|_{L^\infty})\|\nabla^3(\varrho, {\bf u})\|_{L^2}^2,
 \label{es-varrho-H-3-3}
 \end{align}
 and
 \begin{align}\label{es-varrho-H-3-4}
  &- \left\langle \frac{P'(\varrho +\bar\rho)}{(\varrho +\bar\rho)^2} \nabla^3({\bf u} \cdot\nabla\varrho)^H,\nabla^3\varrho^H\right\rangle
  \notag\\
  &
  = - \left\langle \frac{P'(\varrho +\bar\rho)}{(\varrho +\bar\rho)^2} \nabla^3({\bf u} \cdot\nabla\varrho),\nabla^3\varrho\right\rangle
   + \left\langle \frac{P'(\varrho +\bar\rho)}{(\varrho +\bar\rho)^2} \nabla^3({\bf u} \cdot\nabla\varrho)^L,\nabla^3\varrho\right\rangle
   \notag\\
   &
   \quad +\left\langle \frac{P'(\varrho +\bar\rho)}{(\varrho +\bar\rho)^2} \nabla^3({\bf u} \cdot\nabla\varrho)^H,\nabla^3\varrho^L\right\rangle
  \\
  &
  \le - \left\langle \frac{P'(\varrho +\bar\rho)}{(\varrho +\bar\rho)^2} [\nabla^3, {\bf u}] \cdot\nabla\varrho,\nabla^3\varrho\right\rangle +\left\langle {\rm div}\left(\frac{P'(\varrho +\bar\rho)}{(\varrho +\bar\rho)^2} {\bf u}\right),\frac12|\nabla^3\varrho|^2\right\rangle
  \notag\\
  &
  \quad +C\|\nabla^3({\bf u} \cdot\nabla\varrho)^L\|_{L^2}\|\nabla^3\varrho\|_{L^2}
  +C\|\nabla^2({\bf u} \cdot\nabla\varrho)^H\|_{L^2}\|\nabla^4\varrho^L\|_{L^2}
  \notag\\&\le C(\|{\bf u}\|_{L^\infty} + \|\nabla\varrho\|_{L^3} +\|\nabla(\varrho, {\bf u})\|_{L^\infty})\|\nabla^3(\varrho, {\bf u})\|_{L^2}^2.\notag
 \end{align}
 Thus plugging \eqref{es-varrho-H-3-2}--\eqref{es-varrho-H-3-4} into \eqref{es-varrho-H-3-1} yields \eqref{es-varrho-H-3}.

 Taking $\langle \nabla^3 \eqref{2eq-var-H}_2, \nabla^3 {\bf u}^H\rangle$, we can obtain
 \begin{equation}\label{es-u-H-3-1}
 \begin{split}
  &\frac12\frac{d}{dt}\|\nabla^3 {\bf u}^H\|_{L^2}^2 +\left\langle\nabla^3\left(\frac{P'(\varrho +\bar\rho)}{\varrho +\bar\rho}\nabla\varrho\right)^H, \nabla^3 {\bf u}^H\right\rangle +\alpha\|\nabla^3 {\bf u}^H\|_{L^2}^2 -\beta\langle\nabla^4\varphi^H, \nabla^3 {\bf u}^H\rangle
  \\
  &= -\langle\nabla^3({\bf u}\cdot\nabla {\bf u})^H, \nabla^3 {\bf u}^H\rangle.
  \end{split}
  \end{equation}
 As in the proof of \eqref{es-u-k-2} and \eqref{es-varrho-H-3-3}, we have
 \begin{align}
 &\left\langle\nabla^3\left(\frac{P'(\varrho +\bar\rho)}{\varrho +\bar\rho}\nabla\varrho\right)^H, \nabla^3 {\bf u}^H\right\rangle
 \nonumber
 \\
 &=
  \left\langle\nabla^3\left(\frac{P'(\varrho +\bar\rho)}{\varrho +\bar\rho}\nabla\varrho\right), \nabla^3 {\bf u}\right\rangle
  - \left\langle\nabla^3\left(\frac{P'(\varrho +\bar\rho)}{\varrho +\bar\rho}\nabla\varrho\right)^L, \nabla^3 {\bf u}\right\rangle
  - \left\langle\nabla^3\left(\frac{P'(\varrho +\bar\rho)}{\varrho +\bar\rho}\nabla\varrho\right)^H, \nabla^3 {\bf u}^L\right\rangle
 \nonumber
  \\
  &\geq \left\langle \frac{P'(\varrho +\bar\rho)}{\varrho +\bar\rho}\nabla^{4}\varrho, \nabla^3 {\bf u}\right\rangle
  - \left\langle \left[\nabla^3, \frac{P'(\varrho +\bar\rho)}{\varrho +\bar\rho}\right]\nabla\varrho, \nabla^3{\bf u}\right\rangle
  - \frac{P'(\bar\rho)}{\bar\rho}\left\langle\nabla^3\left(\nabla\varrho\right)^L, \nabla^3 {\bf u}\right\rangle
  \nonumber
 \\
  &
  \quad - \left\langle\nabla^3\left(\left(\frac{P'(\varrho +\bar\rho)}{\varrho +\bar\rho} -\frac{P'(\bar\rho)}{\bar\rho}\right)\nabla\varrho\right)^L, \nabla^3 {\bf u}\right\rangle
  - \frac{P'(\bar\rho)}{\bar\rho}\left\langle\nabla^3\left(\nabla\varrho\right)^H, \nabla^3 {\bf u}^L\right\rangle
  \nonumber
 \\
  &
  \quad- \left\langle\nabla^3\left(\left(\frac{P'(\varrho +\bar\rho)}{\varrho +\bar\rho} -\frac{P'(\bar\rho)}{\bar\rho}\right)\nabla\varrho\right)^H, \nabla^3 {\bf u}^L\right\rangle
 \nonumber
  \\
  &\geq \left\langle \frac{P'(\varrho +\bar\rho)}{\varrho +\bar\rho}\nabla^{4}\varrho, \nabla^3 {\bf u}\right\rangle
  - \frac{P'(\bar\rho)}{\bar\rho}\left\langle\nabla^4\varrho^L, \nabla^3 {\bf u}\right\rangle
  - \frac{P'(\bar\rho)}{\bar\rho}\left\langle\nabla^4\varrho^H, \nabla^3 {\bf u}^L\right\rangle
  - C\|\nabla\varrho\|_{L^\infty}\|\nabla^3(\varrho, {\bf u})\|_{L^2}^2,
  \label{es-u-H-3-2}
  \end{align}
 and
 \begin{equation}\label{es-u-H-3-3}
 \begin{split}
  &-\langle\nabla^3({\bf u}\cdot\nabla {\bf u})^H, \nabla^3 {\bf u}^H\rangle
  \\
  &
  = - \langle\nabla^3({\bf u}\cdot\nabla {\bf u}), \nabla^3 {\bf u}\rangle
     + \langle\nabla^3({\bf u}\cdot\nabla {\bf u})^L, \nabla^3 {\bf u}\rangle
     + \langle\nabla^3({\bf u}\cdot\nabla {\bf u})^H, \nabla^3 {\bf u}^L\rangle
   \\
  & = -\langle[\nabla^3, {\bf u}]\cdot\nabla {\bf u}, \nabla^3 {\bf u}\rangle +\left\langle{\rm div}{\bf u}, \frac12|\nabla^3 {\bf u}|^2\right\rangle
     + \langle\nabla^3({\bf u}\cdot\nabla {\bf u})^L, \nabla^3 {\bf u}\rangle
     - \langle\nabla^2({\bf u}\cdot\nabla {\bf u})^H, \nabla^3{\rm div}{\bf u}^L\rangle
  \\
  &\le C\left(\|u\|_{L^\infty}+\|\nabla {\bf u}\|_{L^3\cap L^\infty}\right)\|\nabla^3 {\bf u}\|_{L^2}^2.
 \end{split}
 \end{equation}
 Hence by combining \eqref{es-u-H-3-1} with \eqref{es-u-H-3-2}--\eqref{es-u-H-3-3} and using the a priori assumption \eqref{E-t-priori}, we can deduce \eqref{es-u-H-3}.

 Taking $\langle \nabla^3 \eqref{2eq-var-H}_3, \nabla^3 \varphi^H\rangle$ yields \eqref{es-varphi-H-3} by integration by parts.  And $\langle \nabla^3 \eqref{2eq-var-H}_1, \nabla^2 {\bf u}^H\rangle +\langle \nabla^2\eqref{2eq-var-H}_2, \nabla^3 \varrho^H\rangle$ yields
\begin{equation}\label{es-varrho-u-H-3-1}
\begin{split}
 &\frac{d}{dt}\langle\nabla^3\varrho^H, \nabla^2 {\bf u}^H\rangle
 +\langle\nabla^3((\varrho +\bar\rho){\rm div}{\bf u})^H, \nabla^2 {\bf u}^H\rangle
 +\left\langle\nabla^2\left(\frac{P'(\varrho +\bar\rho)}{\varrho +\bar\rho}\nabla\varrho\right)^H, \nabla^3\varrho^H\right\rangle
 \\
 &
 \quad+\alpha\langle\nabla^2 {\bf u}^H, \nabla^3\varrho^H\rangle -\beta\langle\nabla^3\varphi^H, \nabla^3\varrho^H\rangle
 \\
 &
 = -\langle\nabla^3({\bf u}\cdot\nabla\varrho)^H, \nabla^2 {\bf u}^H\rangle
 -\langle\nabla^2({\bf u}\cdot\nabla {\bf u})^H, \nabla^3\varrho^H\rangle,
\end{split}
\end{equation}
\begin{equation}\label{es-varrho-u-H-3-2}
\begin{split}
 \langle\nabla^3((\varrho +\bar\rho){\rm div}{\bf u})^H, \nabla^2 {\bf u}^H\rangle
 &
 = -\langle\nabla^2((\varrho +\bar\rho){\rm div}{\bf u})^H, \nabla^2{\rm div}{\bf u}^H\rangle
 \\
 &
 = -\bar\rho\langle\nabla^2{\rm div}{\bf u}^H, \nabla^2{\rm div}{\bf u}^H\rangle
 -\langle\nabla^2(\varrho {\rm div}{\bf u})^H, \nabla^2{\rm div}{\bf u}^H\rangle
 \\
 &
 \geq -\bar\rho\|\nabla^2{\rm div}{\bf u}^H\|_{L^2}^2 - C(\|\varrho\|_{L^\infty} +\|\nabla {\bf u}\|_{L^3})\|\nabla^3 (\varrho, {\bf u})\|_{L^2}^2,
 \end{split}
\end{equation}
\begin{equation}\label{es-varrho-u-H-3-3}
\begin{split}
 &\left\langle\nabla^2\left(\frac{P'(\varrho +\bar\rho)}{\varrho +\bar\rho}\nabla\varrho\right)^H, \nabla^3\varrho^H\right\rangle
 \\
 &
 = \frac{P'(\bar\rho)}{\bar\rho}\left\langle\nabla^2\left(\nabla\varrho\right)^H, \nabla^3\varrho^H\right\rangle
  + \left\langle\nabla^2\left(\left(\frac{P'(\varrho +\bar\rho)}{\varrho +\bar\rho} -\frac{P'(\bar\rho)}{\bar\rho}\right)\nabla\varrho\right)^H, \nabla^3\varrho^H\right\rangle
  \\
  &
  \geq  \frac{P'(\bar\rho)}{\bar\rho}\|\nabla^3\varrho^H\|_{L^2}^2 - C(\|\varrho\|_{L^\infty} +\|\nabla\varrho\|_{L^3})\|\nabla^3\varrho\|_{L^2}^2,
 \end{split}
\end{equation}

\begin{equation}\label{es-varrho-u-H-3-4}
 \alpha\langle\nabla^2 {\bf u}^H, \nabla^3\varrho^H\rangle
  -\beta\langle\nabla^3\varphi^H, \nabla^3\varrho^H\rangle
 \le \frac{P'(\bar\rho)}{2\bar\rho}\|\nabla^3\varrho^H\|_{L^2}^2
 +\frac{\alpha^2\bar\rho}{P'(\bar\rho)}\|\nabla^2 {\bf u}^H\|_{L^2}^2 +\frac{\beta^2\bar\rho}{P'(\bar\rho)}\|\nabla^3\varphi^H\|_{L^2}^2,
\end{equation}
and
\begin{equation}\label{es-varrho-u-H-3-5}
\begin{split}
 -\langle\nabla^3({\bf u}\cdot\nabla\varrho)^H, \nabla^2 {\bf u}^H\rangle
 -\langle\nabla^2({\bf u}\cdot\nabla {\bf u})^H, \nabla^3\varrho^H\rangle
 &
 = \langle\nabla^2({\bf u}\cdot\nabla\varrho)^H, \nabla^2{\rm div}{\bf u}^H\rangle
 -\langle\nabla^2({\bf u}\cdot\nabla {\bf u})^H, \nabla^3\varrho^H\rangle
 \\
 &
 \le C(\|{\bf u}\|_{L^\infty} +\|\nabla(\varrho, {\bf u})\|_{L^3})\|\nabla^3(\varrho, {\bf u})\|_{L^2}^2.
\end{split}
\end{equation}
Hence plugging \eqref{es-varrho-u-H-3-2}--\eqref{es-varrho-u-H-3-5} into \eqref{es-varrho-u-H-3-1} implies \eqref{es-varrho-u-H-3}.

 Taking $\langle \nabla^2 \eqref{2eq-var-H}_2, \nabla^3 \varphi^H\rangle +\langle \nabla^3\eqref{2eq-var-H}_3, \nabla^2 {\bf u}^H\rangle$ gives rise to
\begin{equation}\label{es-u-varphi-H-3-1}
\begin{split}
 &\frac{d}{dt}\langle\nabla^2 {\bf u}^H, \nabla^3\varphi^H\rangle
 + \left\langle\nabla^2\left(\frac{P'(\varrho +\bar\rho)}{\varrho +\bar\rho}\nabla\varrho\right)^H, \nabla^3\varphi^H\right\rangle
 +(\alpha +\nu)\langle\nabla^2 {\bf u}^H, \nabla^3\varphi^H\rangle
  -\beta\|\nabla^3\varphi^H\|_{L^2}^2
 \\
 &
 +\mu\langle\nabla^4\varphi^H, \nabla^3 {\bf u}^H\rangle
 -\gamma\langle\nabla^3\varrho^H, \nabla^2 {\bf u}^H\rangle
 \\
 &
 = - \langle\nabla^2({\bf u}\cdot\nabla {\bf u})^H, \nabla^3\varphi^H\rangle.
\end{split}
\end{equation}
Similarly, we have
\begin{align}\label{es-u-varphi-H-3-2}
 &\left\langle\nabla^2\left(\frac{P'(\varrho +\bar\rho)}{\varrho +\bar\rho}\nabla\varrho\right)^H, \nabla^3\varphi^H\right\rangle
 \notag\\
 &
 = \frac{P'(\bar\rho)}{\bar\rho} \left\langle\nabla^2\nabla\varrho^H, \nabla^3\varphi^H\right\rangle
  + \left\langle\nabla^2\left(\left(\frac{P'(\varrho +\bar\rho)}{\varrho +\bar\rho} -\frac{P'(\bar\rho)}{\bar\rho} \right)\nabla\varrho\right)^H, \nabla^3\varphi^H\right\rangle
 \\
 &
 \geq \frac{P'(\bar\rho)}{\bar\rho} \left\langle\nabla^3\varrho^H, \nabla^3\varphi^H\right\rangle
  -C(\|\varrho\|_{L^\infty} +\|\nabla\varrho\|_{L^3})\|\nabla^3(\varrho, \varphi)\|_{L^2}^2,\notag
\end{align}
then by combining \eqref{es-u-varphi-H-3-1} with \eqref{es-u-varphi-H-3-2} and making use of the Cauchy inequality, we can obtain \eqref{es-u-varphi-H-3}. This completes the proof of Lemma \ref{a priori-sol-H}.
\end{proof}
\section{Proof of Nonlinear instability}\label{S4}

We mention that the local existence of solutions to the system can be established by using the standard iteration argument as in \cite{K, KK} and the details are omitted. Furthermore, by the estimate in Lemma \ref{es-nonli-instab-en}, we can conclude the following proposition.
\begin{Proposition}\label{nonli-insta2}
 Assume that the notations and hypotheses in Theorem \ref{2mainth} are in force. For any given initial data $(\widehat{\varrho_0}, \widehat{d_0}, \widehat{\varphi_0}) \in H^3(\mathbb{R}^3)$, there exists a $T>0$ and a unique solution $(\varrho, {\bf u}, \varphi) \in C^0([0, T]; H^3(\mathbb{R}^3))$ to the Cauchy problem \eqref{2eq-var}. Moreover, there is a constant $\bar\delta_0\in (0,1]$, such that if $\mathcal{E}(t) \le \bar\delta_0$ on $[0,T]$, then the solution satisfies
 \begin{equation}\label{insta-E1}
   \mathcal{E}^2(t) \le C\mathcal{E}^2(0) + C\int_0^t \|(\varrho, {\bf u}, \varphi)(\tau)\|_{L^2}^2d\tau,
 \end{equation}
 where $\mathcal{E}(t)$ is defined in \eqref{E-t-priori}.
 \end{Proposition}

Now, we are in a position to prove Theorem \ref{2mainth} by adopting the basic ideas in \cite{GS, JT, JJ, JJW, WT}.

\underline{Proof of Theorem \ref{2mainth}}.  By virtue of Proposition \ref{li-insta2}, we can find a linear solution
$(\varrho^l_{\bar\Theta}, {\bf u}^l_{\bar\Theta}, \varphi^l_{\bar\Theta})$ satisfying
 \begin{equation}\nonumber
  \varrho^l_{\bar\Theta} = \mathcal{F}^{-1}\left[e^{t\lambda_0(|\xi|)}\widehat{\varrho_{0,\bar\Theta}^l}\right], \quad {\bf u}_{\bar\Theta}^l = \Lambda^{-1}\nabla\mathcal{F}^{-1}\left[e^{t\lambda_0(|\xi|)}\widehat{d_{0,\bar\Theta}^l}\right] \quad\text{and}\quad \varphi^l_{\bar\Theta} =  \mathcal{F}^{-1}\left[e^{t\lambda_0(|\xi|)}\widehat{\varphi_{0,\bar\Theta}^l}\right],
 \end{equation}
 which solves the linear system \eqref{eq-li} with the initial data $(\varrho_{0,\bar\Theta}^l, {\bf u}_{0,\bar\Theta}^l, \varphi_{0,\bar\Theta}^l)$ constructed in \eqref{li-varrho-u-varphi-1} satisfying
 \begin{equation}\nonumber
  \|\varrho_{0,\bar\Theta}^l\|_{L^2} \|{\bf u}_{0,\bar\Theta}^l\|_{L^2} \|\varphi_{0,\bar\Theta}^l\|_{L^2} >0,
 \end{equation}
 and obviously $(\varrho_{0,\bar\Theta}^l, {\bf u}_{0,\bar\Theta}^l, \varphi_{0,\bar\Theta}^l) \in H^3(\mathbb{R}^3)$.

 Denote $(\varrho_{0,\bar\Theta}^\delta, {\bf u}_{0,\bar\Theta}^\delta, \varphi_{0,\bar\Theta}^\delta) := \delta (\varrho_{0,\bar\Theta}^l, {\bf u}_{0,\bar\Theta}^l, \varphi_{0,\bar\Theta}^l)$ and $C_0:= \|(\varrho_{0,\bar\Theta}^l, {\bf u}_{0,\bar\Theta}^l, \varphi_{0,\bar\Theta}^l)\|_3$. Due to Proposition \ref{nonli-insta2}, there is a $\widetilde{\delta}\in (0,1)$, such that for any $\delta <\widetilde{\delta}$, there exists a unique local solution $(\varrho^\delta, {\bf u}^\delta, \varphi^\delta)$ to the system \eqref{2eq-var} with the initial data $(\varrho_{0,\bar\Theta}^\delta, {\bf u}_{0,\bar\Theta}^\delta, \varphi_{0,\bar\Theta}^\delta)$. Let $\bar{\delta}_0$ be the same constants as in Proposition \ref{nonli-insta2} and $\delta_0 = \min\{\widetilde{\delta}, \bar{\delta}_0, \frac1{4C_0},1\}$. For any $\delta \in(0, \delta_0)$, we let
 \begin{equation}\nonumber
  T^\delta = \frac1\Theta \ln\frac{2\epsilon_1}{\delta}, \;\ \text{i.e.}, \;\ \delta e^{\Theta T^\delta} = 2\epsilon_1, \;\ \text{and} \;\ \bar{\Theta} = \min\left\{\frac{\Theta}{2}, \frac{1}{T^\delta}\right\}
 \end{equation}
 where $\epsilon_1 = \min\left\{ \frac{\delta_0}{4C_1\sqrt{C_0^2 + \frac{1 }{2\Theta \delta_0^2}}}, \frac1{C_2^6\delta_0^4}, \frac{32\delta_0^2}{(C_2e)^6}, \frac{32\delta_0^2c_1}{(C_2e)^6}\right\}$ with $C_1$, $C_2$ and $c_1$ to be determined. Moreover, we set
 \begin{equation}\nonumber
  T^* = \sup \left\{ t\in (0, T^{max}) \Bigg | \sup_{0\le \tau\le t} \mathcal{E}((\varrho^\delta, {\bf u}^\delta, \varphi^\delta)(\tau)) \le \delta_0\right\},
 \end{equation}
and
 \begin{equation}\nonumber
  T^{**} = \sup \left\{ t\in (0, T^{max}) \Bigg | \sup_{0\le \tau\le t} \|(\varrho^\delta, {\bf u}^\delta, \varphi^\delta)(\tau))\|_{L^2}\le \delta \delta_0^{-1} e^{\Theta t}\right\},
 \end{equation}
 where $T^{max}$ denotes the maximal time of existence. Obviously, $T^* T^{**}>0$, and we have
 \begin{equation}\label{T*2}
  \mathcal{E}((\varrho^\delta, {\bf u}^\delta, \varphi^\delta)(T^*)) = \delta_0 \; \text{ if }\; T^* <\infty,
 \end{equation}
 and
 \begin{equation}\label{T**2}
  \|(\varrho^\delta, {\bf u}^\delta, \varphi^\delta)(T^{**})\|_{L^2} = \delta \delta_0^{-1} e^{\Theta t} \; \text{ if }\; T^{**} < T^{max}.
 \end{equation}

 We claim that
  \begin{equation}\label{T-delta2}
  T^\delta = \min \{T^\delta, T^*, T^{**}\}.
 \end{equation}
 If $T^* = \min\{T^\delta, T^*, T^{**}\}$, then $T^*\le T^\delta <\infty$.  Hence by \eqref{insta-E1}, there exists one positive constant $C_1$, such that
 \begin{equation}\nonumber
  \begin{split}
  \mathcal{E}((\varrho^\delta, {\bf u}^\delta, \varphi^\delta)(T^*))^2
  &
  \le C\mathcal{E}((\varrho^\delta, {\bf u}^\delta, \varphi^\delta)(0))^2 + C\int_0^{T^*} \|(\varrho^\delta, {\bf u}^\delta, \varphi^\delta)(t)\|_{L^2}^2dt
  \\
  &
  \le CC_0^2\delta^2 +C\int_0^{T^*}\left(\delta \delta_0^{-1} e^{\Theta t}\right)^2dt
  \\
  &
  \le \left(CC_0^2 + \frac{C e^{2\Theta T^*}}{2\Theta \delta_0^{2} }\right)\delta^2
  \\
  &
  \le C_1^2\left(C_0^2 + \frac{1 }{2\Theta \delta_0^2}\right)e^{2\Theta T^\delta}\delta^2
  \\
  &
  \le C_1^2\left(C_0^2 + \frac{1 }{2\Theta \delta_0^2}\right)(2\epsilon_1)^2,
  \end{split}
  \end{equation}
  which implies
  \begin{equation}\nonumber
    \mathcal{E}((\varrho^\delta, {\bf u}^\delta, \varphi^\delta)(T^*))
    \le C_1\sqrt{C_0^2 + \frac{1 }{2\Theta \delta_0^2}}\times 2\epsilon_1 \le \frac{\delta_0}2 < \delta_0.
   \end{equation}
   This contradicts \eqref{T*2}. If $T^{**} = \min\{T^\delta, T^*, T^{**}\}$, then $T^{**} < T^{max}$. Hence we have from \eqref{so-ex}, Propositions \ref{pro-L2-theory}--\ref{li-insta2} and \eqref{T**2} that
 \begin{align}\label{T**3}
  \|(\varrho^\delta, {\bf u}^\delta, \varphi^\delta)(T^{**})\|_{L^2}
  &
  = \|(\widehat{\varrho^\delta}, \widehat{{\bf u}^\delta}, \widehat{\varphi^\delta})(T^{**})\|_{L^2}
  = \|(\widehat{\varrho^\delta}, \widehat{d^\delta}, \widehat{\varphi^\delta})(T^{**})\|_{L^2}
  \notag\\
  &
  \le
  \|\delta(\widehat{\varrho^l_{\bar\Theta}}, \widehat{{\bf u}^l_{\bar\Theta}}, \widehat{\varphi^l_{\bar\Theta}})(T^{**})\|_{L^2}
     + C\!\!\int_0^{T^{**}}\!\!\left\|e^{(T^{**} - t) A(|\xi|)}\mathcal{F}[\mathcal{N}(\varrho^\delta, {\bf u}^\delta)] \right\|_{L^2}dt
  \notag\\
  &
  \le \delta \|e^{\lambda_0(|\xi|)T^{**}}(\widehat{\varrho_{0,\bar\Theta}^l}, \widehat{{\bf u}_{0,\bar\Theta}^l}, \widehat{\varphi_{0,\bar\Theta}^l})\|_{L^2}
  \notag\\
  &\quad +C\int_0^{T^{**}} e^{\Theta(T^{**} - t)}\|(\varrho^\delta, {\bf u}^\delta)(t)\|_{L^2}\|\nabla(\varrho^\delta, {\bf u}^\delta)(t)\|_{L^\infty}dt
  \notag\\
  &
  \le C_0\delta e^{\Theta T^{**}} +C\int_0^{T^{**}} e^{\Theta(T^{**} - t)}\|(\varrho^\delta, {\bf u}^\delta)(t)\|_{L^2}^\frac76\|\nabla^3(\varrho^\delta, {\bf u}^\delta)(t)\|_{L^2}^\frac56dt
  \\
  &
  \le C_0\delta e^{\Theta T^{**}} +C\int_0^{T^{**}} e^{\Theta(T^{**} - t)}\delta^\frac76\delta_0^{-\frac76} e^{\frac76\Theta t} \delta_0^\frac56dt
  \notag\\
  &
  \le C_0\delta e^{\Theta T^{**}} +Ce^{\Theta T^{**}}\delta^\frac76\delta_0^{-\frac13}\frac{e^{\frac{\Theta T^{**}}{6} }}{\frac\Theta 6}
  \notag\\
  &
  :=\frac12\delta\delta_0^{-1}e^{\Theta T^{**}}\left(2C_0\delta_0 +12C\Theta\delta^\frac16\delta_0^\frac23 e^{\frac{\Theta T^\delta}{6}}\right)
 \notag \\
  &
\le\frac12\delta\delta_0^{-1}e^{\Theta T^{**}}\left(2C_0\delta_0 + \frac{C_2}2\delta_0^\frac23\epsilon_1^\frac16\right)\le\frac12\delta\delta_0^{-1}e^{\Theta T^{**}},\notag
 \end{align}
   where $\delta_0 \le \frac{1}{4C_0}$ and the definition of $\epsilon_1$ are used. This contradicts \eqref{T**2}. Hence, \eqref{T-delta2} holds.

 At last, by using the similar estimates in \eqref{T**3}, we have from \eqref{so-ex}, Proposition \ref{pro-L2-theory} and Proposition \ref{li-insta2} that
 \begin{equation}\label{lower-varrho22}
  \begin{split}
  \|\varrho^\delta(T^\delta)\|_{L^2}
  &
  \geq
  \|\delta\widehat{\varrho^l_{\bar\Theta}}(T^{\delta})\|_{L^2}
     - C\int_0^{T^\delta}\left\|e^{(T^{\delta} - t) A(|\xi|)}\mathcal{F}[\mathcal{N}(\varrho^\delta, {\bf u}^\delta)]\right\|_{L^2}dt
   \\
   &
   \geq \delta e^{(\Theta - \widetilde{\Theta})T^\delta}\|\widehat{\varrho_{0,\bar\Theta}^l}\|_{L^2} - \frac{C_2}4e^{\frac{7\Theta }{6}T^\delta}\delta^{\frac76} \delta_0^{-\frac13}
   \\
   &
   \geq\delta e^{\Theta T^\delta}\left(e^{-\widetilde{\Theta} T^\delta} -\frac{C_2}{4}e^{\frac{\Theta}{6}T^\delta}\delta^{\frac16}\delta_0^{-\frac13}\right)
   \\
   &
    \geq\delta e^{\Theta T^\delta}\left(e^{- 1} -\frac{C_2}{4}e^{\frac{\Theta}{6}T^\delta}\delta^{\frac16}\delta_0^{-\frac13}\right)
   \\
   &
    = 2\epsilon_1\left(e^{- 1} -\frac{C_2}{4}(2\epsilon_1)^\frac16\delta_0^{-\frac13}\right)
   \\
   &\geq 2\epsilon_1\times\frac12e^{- 1} = \frac{\epsilon_1}{e},
   \end{split}
   \end{equation}
   where $\widetilde{\Theta} \le \frac{1}{T^\delta}$ and the fact that $\epsilon_1 \le \frac{32\delta_0^2}{(C_2e)^6} $ are used. Moreover, from \eqref{li-varrho-u-varphi-1}, we have
   \begin{equation}\nonumber
    \|\widehat{d_{0,\bar\Theta}^l}\|_{L^2}
    = \left\|-\frac{\lambda_0(|\xi|) \Psi(\xi)}{a|\xi|\|\Psi\|_{L^2}} \right\|_{L^2}
    \geq \inf_{|\xi - \xi_0| <2\bar\zeta}\left|-\frac{\lambda_0(|\xi|)}{a|\xi|}\right|
    \geq \frac{\Theta}{a|\xi_0|}
   \end{equation}
   and
   \begin{equation}\nonumber
    \|\widehat{\varphi_{0,\bar\Theta}^l}\|_{L^2}
    = \left\|\frac{\gamma \Psi(\xi)}{\left(\lambda_0(|\xi|) + \nu + \mu|\xi|^2\right)\|\Psi\|_{L^2}} \right\|_{L^2}
    \geq \inf_{|\xi - \xi_0| <2\bar\zeta}\left|\frac{\gamma }{\lambda_0(|\xi|) + \nu + \mu|\xi|^2}\right|
    \geq \frac{\gamma }{\Theta + \nu + 4\mu|\xi_0|^2}.
    \end{equation}
  Set $c_1 = c_1(\Theta, |\xi_0|, \gamma, \nu, \mu) = \min\left\{\frac{\Theta}{a|\xi_0|}, \frac{\gamma }{\Theta + \nu + 4\mu|\xi_0|^2}\right\}$. Therefore, as in the proof of \eqref{lower-varrho22}, we can conclude that
  \begin{equation}\nonumber
   \|{\bf u}^\delta(T^\delta)\|_{L^2}, \quad \|\varphi^\delta(T^\delta)\|_{L^2} \geq \frac{c_1\epsilon_1}{e}.
  \end{equation}
   This completes the proof of Theorem \ref{2mainth} by taking $\displaystyle \epsilon_0 = \min\left\{\frac{\epsilon_1}{e}, \frac{c_1\epsilon_1}{e}\right\}$.
   $\hfill\square$

\section{Proof of Global Existence and Upper Decay Estimates}\label{S5}

This section is devoted to prove the global existence and upper decay rate of the solution stated in Theorem \ref{3mainth}.
It is well--known that the global existence of solutions can be established by combining the local existence result with a priori estimates. The local strong solutions can be obtained by a standard argument as in \cite{K, KK}, then by using the standard continuity argument, global solutions can be proved by the local existence and the a priori estimate.

\begin{Proposition}[A priori estimate]\label{priori1}
  Under the assumption of Lemma \ref{a priori-sol} and $\frac{\nu P'(\bar\rho)}{\gamma \bar\rho} -\beta >0$, it holds that
 \begin{equation}\label{es-priori1}
 \mathcal{E}(t)^2
  +\int_0^t \left(\|(\nu\varphi -\gamma\varrho)(\tau)\|_{L^2}^2 + \left\|\nabla \varrho(\tau)\right\|_2^2 +
  \|{\bf u}(\tau)\|_3^2 +\|\nabla \varphi(\tau)\|_3^2\right) d\tau
  \le C\mathcal{E}(0)^2.
 \end{equation}
 \end{Proposition}

\begin{proof}
Take $K_3 = \min\left\{\frac{\min\left\{\frac{\alpha}{4}, \frac{\beta}{4\mu}\left(\frac{\nu P'(\bar\rho)}{\gamma \bar\rho} -\beta\right)\right\}}{\frac{\bar\rho P'(\bar\rho)}{4\mu\nu} +\frac{(2\alpha^2 + \beta^2)\bar\rho}{P'(\bar\rho)} +4\bar\rho}, \sqrt{\frac{\alpha P'(\bar\rho)}{2\mu\bar\rho^2}}\right\}$. Then making the summation \eqref{es-u-varphi-ener1} $+K_3\times $\eqref{es-dispper-varphi-ener1} and letting $\delta$ sufficiently small, one can arrive at
 \begin{equation}\label{es-varrho-u-varphi-ener2}
 \begin{split}
  &\frac{d}{dt}\left(\sum_{k=1}^3\left(\frac12\|\nabla^k{\bf u}\|_{L^2}^2 +\frac\beta\mu \langle\nabla^{k-1}{\bf u}, \nabla^k\varphi\rangle +\frac{\beta P'(\bar\rho)}{2\mu\gamma\bar\rho}\|\nabla^k\varphi\|_{L^2}^2 +\frac{\alpha +\nu}{2\mu}\|\nabla^{k-1}{\bf u}\|_{L^2}^2 +\frac{P'(\bar\rho)}{2\bar\rho^2}\|\nabla^k\varrho\|_{L^2}^2 \right.\right.
  \\
  &\qquad\qquad\left.+\frac{1}{2\bar\rho}\left( \frac{\gamma}{\mu}\left(\frac{\nu P'(\bar\rho)}{\gamma \bar\rho} -\beta\right) +\frac{\alpha P'(\bar\rho)}{\mu\bar\rho}\right)\|\nabla^{k-1}\varrho\|_{L^2}^2\right)
  \\
  &\qquad \left.+ \frac{K_3P'(\bar\rho)}{16\mu\gamma^2\bar\rho}\|\nu\varphi - \gamma\varrho\|_{L^2}^2 + K_3\sum_{k=1}^3\langle\nabla^k\varrho, \nabla^{k-1}{\bf u}\rangle\right)
  \\
  &+\sum_{k=1}^3\left(\frac{\alpha}{8}\|\nabla^k{\bf u}\|_{L^2}^2 +\frac{\alpha(\alpha+\nu)}{2\mu}\| \nabla^{k-1}{\bf u}\|_{L^2}^2 +\frac{\beta P'(\bar\rho)}{\gamma\bar\rho}\|\nabla^{k+1} \varphi\|_{L^2}^2 +\frac{\beta}{4\mu}\left(\frac{\nu P'(\bar\rho)}{\gamma \bar\rho} -\beta\right)\|\nabla^k\varphi\|_{L^2}^2 \right)
  \\
  &+K_3\left(\frac{\nu P'(\bar\rho)}{16\mu\gamma^2\bar\rho}\|\nu\varphi - \gamma\varrho\|_{L^2}^2 +\frac{P'(\bar\rho)}{8\bar\rho}\|\nabla\varrho\|_2^2 \right)
  \\
  &\le 0,
  \end{split}
  \end{equation}
 which implies \eqref{es-priori1} by the Cauchy inequality and the assumption condition that $\frac{\nu P'(\bar\rho)}{\gamma \bar\rho} -\beta>0$.
\end{proof}

Next, we intend to obtain the optimal decay rates of the solution to finish the proof of Theorem \ref{3mainth}.
Define the time--weighted energy functional
\begin{equation}\label{M1}
 \mathcal{M}(t) = \sup_{0\le \tau\le t}\left(\sum_{k=0}^3(1+\tau)^{\frac34+\frac k2}\|\nabla^k(\varrho, \varphi)(\tau)\|_0 + \sum_{k=0}^2(1+\tau)^{\frac54 + \frac k2} \|\nabla^k{\bf u}(\tau)\|_0 +(1+\tau)^\frac94  \|\nabla^3 {\bf u}(\tau)\|_0 \right).
\end{equation}
It is noted that we can not derive the optimal decay rate of the 3--order derivatives of the velocity from the definition of $\mathcal{M}(t)$. This is caused by that we fail to estimate the 3--order derivatives of the solution by the spectral analysis since the nonlinear terms involve the 4--order derivatives, however, we can control the 3--order derivatives of the solution by $\|\nabla^2 {\bf u}\|_{L^2}$ via the pure energy estimate.

We shall prove the following proposition to achieve the optimal decay rate part of Theorem \ref{3mainth}.
\begin{Proposition}\label{es-thm-M}
 Under the assumptions of Theorem \ref{3mainth}, it holds
 \begin{equation}\nonumber
  \mathcal{M}(t) \le C(\mathcal{E}_0, K_0).
 \end{equation}
\end{Proposition}

Next, we divide the proof of Proposition \ref{es-thm-M} into the following steps: derive the optimal decay rates on the solution and its highest--order derivatives separately. To this end, we first need some tools to dealt with the integration on the time \cite{DUYZ1}.
\begin{lemma}\label{s1s2}
 Assume $s_1 >1$, $s_2\in [0, s_1]$, then we have
 \begin{equation}\nonumber
  \int_0^t(1+t-\tau)^{-s_1}(1+\tau)^{-s_2}d\tau
  \le C(s_1, s_2)(1+t)^{-s_2}.
 \end{equation}
\end{lemma}

Now, we estimate the decay rate on the lower--frequent part of the solution as:
\begin{lemma}\label{es-fre-op}
 Assume that the assumptions of Proposition \ref{es-thm-M} are in force. Then it holds
 \begin{equation}\label{low-fre-varrho-varphi1}
  \|(\varrho, \varphi)(t)\|_{L^2}
  \le C\left(K_0 + \mathcal{M}^2(t) \right)(1+t)^{-\frac34},
 \end{equation}
 and
 \begin{equation}\label{low-fre-u1}
  \|{\bf u}(t)\|_{L^2}
  \le C\left(K_0+ \mathcal{M}^2(t) \right)(1+t)^{-\frac54}.
 \end{equation}
\end{lemma}
\begin{proof}
First, by using the estimates \eqref{es-varrho-low-k}, \eqref{es-varphi-low-k} and \eqref{es-varrho-d-varphi-high1} in Lemma \ref{li-de-L-H} and the estimates \eqref{decay-S1-low-k1} and \eqref{decay-S1-S2-high-k} in Lemma \ref{prop-decay-nonli}, we have from \eqref{so-ex} that
\begin{align}\label{es-de-varrho-varphi-or-k}
 \|(\varrho, \varphi)\|_{L^2}
 &
 \le \|(\varrho^L, \varphi^L)\|_{L^2} +\|(\varrho^H, \varphi^H)\|_{L^2}
 \notag\\
 &
 \lesssim (1+t)^{-\frac34}\|(\varrho_0^L, d_0^L, \varphi_0^L)\|_{L^1}
 + \int_0^t (1+t-\tau)^{-\frac34}\|(N_1, N_2)(\tau)\|_{L^1}d\tau
 \notag\\
 &
 \quad + e^{-\frac{\alpha}{4}t}\|(\varrho_0^H, d_0^H, \varphi_0^H)\|_{L^2}
 + \int_0^te^{-\frac{\alpha}4 (t-\tau)}\|(N_1, N_2)(\tau)\|_{L^2}d\tau
 \notag\\
 &
 \lesssim (1+t)^{-\frac34}\|(\varrho_0, d_0, \varphi_0)\|_{L^1\cap L^2}
 \notag\\
 &
 \quad
 + \int_0^t(1+t-\tau)^{-\frac34}\|(\varrho, {\bf u})(\tau)\|_{L^2}\|\nabla(\varrho, {\bf u})(\tau)\|_{L^2}d\tau
\notag \\
 &
 \quad +\int_0^t(1+t-\tau)^{-\frac34}\|(\varrho, {\bf u})(\tau)\|_{L^\infty}\|\nabla(\varrho, {\bf u})(\tau)\|_{L^2}d\tau
 \\
 &
 \lesssim (K_0 +\mathcal{E}_0)(1+t)^{-\frac34}
 + \int_0^t(1+t-\tau)^{-\frac34}(1+\tau)^{-\frac34}\mathcal{M}(\tau)(1+\tau)^{-\frac54}\mathcal{M}(\tau)d\tau
 \notag\\
 &
 \quad +\int_0^t(1+t-\tau)^{-\frac34}\|(\varrho, {\bf u})(\tau)\|_{L^2}^\frac14\|\nabla^2(\varrho, {\bf u})(\tau)\|_{L^2}^\frac34\|\nabla(\varrho, {\bf u})(\tau)\|_{L^2}d\tau
 \notag\\
 &
 \lesssim (K_0 +\mathcal{E}_0)(1+t)^{-\frac34}
 + \mathcal{M}^2(t)\int_0^t(1+t-\tau)^{-\frac34}(1+\tau)^{-2}d\tau
 \notag\\
 &
 \quad + \mathcal{M}^2(t)\int_0^t(1+t-\tau)^{-\frac34}(1+\tau)^{-\frac34\times\frac14 - \frac74\times\frac34}(1+\tau)^{-\frac54}d\tau\notag
 \\
 &\lesssim (1+t)^{-\frac34}\left(K_0 +\mathcal{E}_0 +\mathcal{M}^2(t)\right),\notag
 \end{align}
 where the H\"older inequality, Lemma \ref{s1s2}, the definition \eqref{M1} and the monotonicity of $\mathcal{M}(t)$ are used. This yields \eqref{low-fre-varrho-varphi1}. Similarly, we can derive \eqref{low-fre-u1} from \eqref{so-ex} and Lemma \ref{li-de-L-H}--Lemma \ref{prop-decay-nonli}, or we can refer to the details in the proof of next lemma.
\end{proof}

\begin{lemma}\label{le-es-loworder1}
 Assume that the assumptions of Proposition \ref{es-thm-M} are in force. Then, it holds
 \begin{equation}\nonumber
  \|\nabla^2 {\bf u}(t)\|_{L^2}
  \le C(1+t)^{-\frac94}\left(K_0 +\mathcal{E}_0 + \mathcal{M}^2(t)\right).
 \end{equation}
\end{lemma}
\begin{proof}
 As in the proof of Lemma \ref{es-fre-op}, by using the estimates \eqref{es-d-low-k}, \eqref{es-varrho-d-varphi-high1} and \eqref{es-Omega-all-k} in Lemma \ref{li-de-L-H} and the estimates \eqref{decay-S2-low-k1} and \eqref{decay-S1-S2-high-k} in Lemma \ref{prop-decay-nonli}, we have from \eqref{so-ex}
 \begin{align}
 \|\nabla^2 {\bf u}\|_{L^2}
 \nonumber
 &\le \|\nabla^2d\|_{L^2} +\|\nabla^2\Omega\|_{L^2}
 \nonumber
 \\
 &
 \le \|\nabla^2d^L\|_{L^2} +\|\nabla^2d^H\|_{L^2} +\|\nabla^2\Omega\|_{L^2}
 \nonumber
 \\
 &
 \lesssim (1+t)^{-\frac94}\|(\varrho_0^L, d_0^L, \varphi_0^L)\|_{L^1}
 + \int_0^\frac t2(1+t-\tau)^{-\frac94}\|(N_1, N_2)(\tau)\|_{L^1}d\tau
 \nonumber
 \\
 &
 \quad +\int_\frac t2^t(1+t-\tau)^{-\frac54}\|\nabla^2(N_1, N_2)(\tau)\|_{L^1}d\tau
\nonumber
 \\
 &
 \quad
 +e^{-\frac{\alpha t}{4}}\|\nabla^2(\varrho_0^H, d_0^H, \varphi_0^H)\|_{L^2} +\int_0^te^{-\frac{\alpha }{4}(t-\tau)}\|\nabla^2(N_1, N_2)(\tau)\|_{L^2}d\tau
 \nonumber
 \\
 &
 \quad
 +e^{-\alpha t}\|\nabla^2\Omega_0\|_{L^2} +\int_0^te^{-\alpha (t-\tau)}\|\nabla^2 N_2(\tau)\|_{L^2}d\tau
 \nonumber
 \\
 &
 \lesssim (1+t)^{-\frac94}\|(\varrho_0, d_0, \Omega_0, \varphi_0)\|_{L^1\cap H^2}
 + \int_0^\frac t2(1+t-\tau)^{-\frac94}\|(N_1, N_2)(\tau)\|_{L^1}d\tau
\nonumber
 \\
 &
 \quad + \int_\frac t2^t(1+t-\tau)^{-\frac54}\|\nabla^2(N_1, N_2)(\tau)\|_{L^1}d\tau
 + \int_0^te^{-\frac{\alpha }{4}(t-\tau)}\|\nabla^2(N_1, N_2)(\tau)\|_{L^2}d\tau
 \label{es-de-u-2-1}
 \\
 &
 \lesssim (1+t)^{-\frac94}(K_0 +\mathcal{E}_0)
 + \int_0^\frac t2(1+t-\tau)^{-\frac94}\|(\varrho, {\bf u})(\tau)\|_{L^2}\|\nabla(\varrho, {\bf u})(\tau)\|_{L^2}d\tau
 \nonumber
 \\
 &
 \quad + \int_\frac t2^t(1+t-\tau)^{-\frac54}\|(\varrho, {\bf u})(\tau)\|_{L^2}\|\nabla^3(\varrho, {\bf u})(\tau)\|_{L^2}d\tau
 \nonumber
 \\
 &
 \quad + \int_0^te^{-\frac{\alpha }{4}(t-\tau)}(\|(\varrho, {\bf u})(\tau)\|_{L^\infty} +\|\nabla(\varrho, {\bf u})(\tau)\|_{L^3})\|\nabla^3(\varrho, {\bf u})(\tau)\|_{L^2}d\tau
 \nonumber
 \\
 &
 \lesssim (1+t)^{-\frac94}(K_0 +\mathcal{E}_0)
 + \mathcal{M}^2(t)\int_0^\frac t2(1+t-\tau)^{-\frac94}(1+\tau)^{-2}d\tau
 \nonumber
 \\
 &
 \quad + \mathcal{M}^2(t)\int_\frac t2^t(1+t-\tau)^{-\frac54}(1+\tau)^{-3}d\tau
 + \mathcal{M}^2(t)\int_0^t(1+t-\tau)^{-\frac{13}4}(1+\tau)^{-\frac{13}4}d\tau
 \nonumber
 \\
 &
 \lesssim (1+t)^{-\frac94}(K_0 +\mathcal{E}_0)
 + \mathcal{M}^2(t)(1+t)^{-\frac94}\int_0^\frac t2(1+\tau)^{-2}d\tau
 \nonumber
 \\
 &
 \quad + \mathcal{M}^2(t)(1+t)^{-3}\int_\frac t2^t(1+t-\tau)^{-\frac54}d\tau
 + \mathcal{M}^2(t)(1+t)^{-\frac{13}4}
 \nonumber
 \\
 &
 \lesssim (1+t)^{-\frac94}(K_0 +\mathcal{E}_0 +\mathcal{M}^2(t)).
 \nonumber
 \end{align}
\end{proof}

Now, we turn to estimate the decay rate on the highest--order derivative of the solution, and we state the result in the following:
\begin{lemma}\label{le-es-highorder1}
 Assume that the assumptions of Proposition \ref{es-thm-M} are in force. Then it holds
 \begin{equation}\label{high-fre-varrho-m1}
  \|\nabla^3(\varrho, {\bf u}, \varphi)(t)\|_{L^2}
  \le C(1+t)^{-\frac94}\left(K_0 + \mathcal{E}_0 + \mathcal{M}^2(t)\right).
 \end{equation}
\end{lemma}
\begin{proof}
As in the proof of Lemma but a little modification, we take the summation $\eqref{es-varrho-k} + \eqref{es-u-k} +\frac\beta\mu\times\eqref{es-u-varphi-k} +\frac{\beta P'(\bar\rho)}{\gamma\mu\bar\rho}\times\eqref{es-varphi-k} +K_3\times\eqref{es-varrho-u-k}$ with $k=3$, and then we can deduce
\begin{equation}\nonumber
\begin{split}
 &\frac{d}{dt}\left( \left\langle\frac{P'(\varrho +\bar\rho)}{2(\varrho +\bar\rho)^2}\nabla^3\varrho, \nabla^3\varrho\right\rangle +\frac12\|\nabla^3 {\bf u}\|_{L^2}^2 +\frac\beta{\mu}\langle\nabla^2 {\bf u}, \nabla^3\varphi\rangle +\frac{\beta P'(\bar\rho)}{\gamma\mu\bar\rho}\|\nabla^3\varphi\|_{L^2}^2 +K_3\langle\nabla^3\varrho, \nabla^2 {\bf u}\rangle\right)
 \\
 &
 \quad+\frac\alpha2\|\nabla^3 {\bf u}\|_{L^2}^2 +\frac{\beta P'(\bar\rho)}{\gamma\bar\rho}\|\nabla^4\varphi\|_{L^2}^2 +\frac{\beta}{2\mu}\left(\frac{\nu P'(\bar\rho)}{\gamma\bar\rho} -\beta\right)\|\nabla^3\varphi\|_{L^2}^2
 +\frac{K_3P'(\bar\rho)}{4\bar\rho}\|\nabla^3\varrho\|_{L^2}^2
 \\
 &
 \le C\|\nabla^2 {\bf u}\|_{L^2}^2,
\end{split}
\end{equation}
where the a priori assumption \eqref{E-t-priori} and the definition of $K_3$ in Proposition \ref{priori1} are used.

Define
 \begin{equation}\nonumber
  \begin{split}
  \mathcal{L}_2(t)
  =\left\langle\frac{P'(\varrho +\bar\rho)}{2(\varrho +\bar\rho)^2}\nabla^3\varrho, \nabla^3\varrho\right\rangle +\frac12\|\nabla^3 {\bf u}\|_{L^2}^2 +\frac\beta{\mu}\langle\nabla^2 {\bf u}, \nabla^3\varphi\rangle +\frac{\beta P'(\bar\rho)}{2\gamma\mu\bar\rho}\|\nabla^3\varphi\|_{L^2}^2 +K_3\langle\nabla^3\varrho, \nabla^2 {\bf u}\rangle.
  \end{split}
 \end{equation}
 By the Cauchy inequality and the fact that $\|f^H\|_{L^2}\le \|\nabla f^H\|_{L^2}$, we can get the following equivalent relationship:
 \begin{equation}\label{L2-u-2}
  \mathcal{L}_2(t) +C_3\|\nabla^2 {\bf u}(t)\|_{L^2}^2
  \approx \|\nabla^3(\varrho, {\bf u}, \varphi)\|_{L^2}^2 + \|\nabla^2 {\bf u}(t)\|_{L^2}^2.
 \end{equation}
Moreover, we have
 \begin{equation}\nonumber
  \frac{d}{dt}\mathcal{L}_2(t) +C_4\mathcal{L}_2(t)
  \le C\|\nabla^2 {\bf u}(t)\|_{L^2}^2.
 \end{equation}
 Then by the Gronwall inequality and Lemma \ref{s1s2}, we can arrive at
 \begin{equation}\nonumber
  \begin{split}
  \mathcal{L}_2(t)
  &\le e^{-C_4t}\mathcal{L}_2(0) + C\int_0^t e^{-C_4(t-\tau)}\|\nabla^2 {\bf u}(\tau)\|_{L^2}^2d\tau
  \\
  &
  \le e^{-C_4t}\left(\|\nabla^3(\varrho_0, {\bf u}_0, \varphi_0)\|_{L^2}^2 + \|\nabla^2{\bf u}_0\|_{L^2}^2\right) + C\int_0^t e^{-C_4(t-\tau)}(1+\tau)^{-\frac92}(K_0 +\mathcal{E}_0 + \mathcal{E}^2(\tau))^2d\tau
  \\
  &
  \le (1+ t)^{-\frac92}(K_0 +\mathcal{E}_0 + \mathcal{E}^2(\tau))^2,
  \end{split}
 \end{equation}
 which together with the relationship \eqref{L2-u-2} yields \eqref{high-fre-varrho-m1}.
\end{proof}

\underline{Proof of Proposition \ref{es-thm-M}}. By combining with Lemma \ref{es-fre-op}, Lemma \ref{le-es-loworder1} and Lemma \ref{le-es-highorder1}, using the definition \eqref{M1} of $\mathcal{M}(t)$, and the Sobolev interpolation inequality, we can finally get
 \begin{equation}\nonumber
  \mathcal{M}(t) \le C(\mathcal{E}_0, K_0) + C\mathcal{M}^2(t),
 \end{equation}
 which together with the smallness of $\mathcal{E}_0$ and $K_0$ implies that $\mathcal{M}(t) \le C(\mathcal{E}_0 +K_0)$. This complete the proof of Proposition \ref{es-thm-M}.

Now we turn to state the optimal decay rate of the highest--order derivatives of the velocity.
\begin{Proposition}\label{pro-de-high-u}
 Under the assumptions of Theorem \ref{3mainth}, it holds
 \begin{equation}\nonumber
  \|\nabla^3 {\bf u}\|_{L^2} \le C(\mathcal{E}_0, K_0)(1+t)^{-\frac{11}4}.
 \end{equation}
 \end{Proposition}
  We divide the proof of Proposition \ref{pro-de-high-u} into the following steps: derive the optimal decay rates on the low--frequent part and the high--frequent part of the 3--order derivative of the velocity respectively.

  \begin{lemma}\label{le-de-u-3-low}
   Under the assumptions of Theorem \ref{3mainth}, it holds
   \begin{equation}\label{e-u-3-low}
    \|\nabla^3 {\bf u}^L\|_{L^2}\le C(\mathcal{E}_0, K_0)(1+t)^{-\frac{11}{4}}.
   \end{equation}
  \end{lemma}
  \begin{proof}
 As in the proof of Lemma \ref{es-fre-op}, by using the estimates \eqref{es-d-low-k} and \eqref{es-Omega-all-k} in Lemma \ref{li-de-L-H} and the estimates \eqref{decay-S2-low-k1} and \eqref{decay-S2-low-k2} in Lemma \ref{prop-decay-nonli}, we have from \eqref{so-ex} that
 \begin{equation*}
 \begin{split}
 \|\nabla^3 {\bf u}^L\|_{L^2}
 &\le \|\nabla^3d^L\|_{L^2} +\|\nabla^3\Omega^L\|_{L^2}
 \\
 &
 \lesssim (1+t)^{-\frac{11}4}\|(\varrho_0^L, d_0^L, \varphi_0^L)\|_{L^1}
 + \int_0^\frac t2(1+t-\tau)^{-\frac{11}4}\|(N_1, N_2)(\tau)\|_{L^1}d\tau
 \\
 &
 \quad
 +\int_\frac t2^t(1+t-\tau)^{-\frac54}\|\nabla^3(N_1^L, N_2^L)(\tau)\|_{L^1}d\tau
 \\
 &
 \quad
 + (1+t)^{-\alpha t}\|\nabla^3\Omega_0\|_{L^2}
 + \int_0^te^{-\alpha (t-\tau)}\|\nabla^3(N_1^L, N_2^L)(\tau)\|_{L^2}d\tau
 \\
 &
 \lesssim (1+t)^{-\frac{11}4}(K_0 +\mathcal{E}_0) + \int_0^\frac t2(1+t-\tau)^{-\frac{11}4}(1+\tau)^{-2}C(\mathcal{E}_0, K_0)d\tau
 \\
 &
 \quad + \int_\frac t2^t(1+t-\tau)^{-\frac54}\|\nabla^2(N_1, N_2)\|_{L^1}d\tau
 + \int_0^te^{-\alpha (t-\tau)}\|\nabla^2(N_1, N_2)\|_{L^2}d\tau
 \\
 & \lesssim C(\mathcal{E}_0, K_0)(1+t)^{-\frac{11}4}
 +\int_\frac t2^t(1+t-\tau)^{-\frac54}(1+\tau)^{-3}C(\mathcal{E}_0, K_0) d\tau
  \\
 &
 \quad + \int_0^te^{-\alpha (t-\tau)}(1+\tau)^{-\frac{15}4}C(\mathcal{E}_0, K_0)d\tau
 \\
 &\le C(\mathcal{E}_0, K_0)(1+t)^{-\frac{11}4},
 \end{split}
 \end{equation*}
 where we used Proposition \ref{es-thm-M} and the fact that $\|\nabla f^L\|_{L^2} \lesssim \| f^L\|_{L^2} \lesssim \|f\|_{L^2}$ again.
  \end{proof}

 \begin{lemma}\label{le-de-u-3-high}
   Under the assumptions of Theorem \ref{3mainth}, it holds
   \begin{equation}\label{e-u-3-high}
    \|\nabla^3 {\bf u}^H\|_{L^2}\le C(1+t)^{-3}.
   \end{equation}
 \end{lemma}
 \begin{proof}
 First, we estimate the decay rate of $\|\nabla^2 {\bf u}^H\|_{L^2}$. As in the proof of \eqref{es-de-u-2-1}, by Lemmas \ref{li-de-L-H}--\ref{prop-decay-nonli}, we have from \eqref{es-de-u-2-1} that
 \begin{equation}\label{es-de-u-or-2-H}
 \begin{split}
 \|\nabla^2 {\bf u}^H\|_{L^2}
 &
 \le \|\nabla^2d^H\|_{L^2} +\|\nabla^2\Omega^H\|_{L^2}
 \\
 &
 \le
 e^{-\frac{\alpha t}{4}}\|\nabla^2(\varrho_0^H, d_0^H, \varphi_0^H)\|_{L^2} +C\int_0^te^{-\frac{\alpha}{4} (t-\tau)}\|\nabla^2(N_1^H(\tau), N_2^H(\tau))\|_{L^2}d\tau
 \\
 &
 \quad
 +e^{-\alpha t}\|\nabla^2\Omega_0\|_{L^2} +C\int_0^te^{-\alpha (t-\tau)}\|\nabla^2 N_2^H(\tau)\|_{L^2}d\tau
 \\
 &
 \le e^{-\frac{\alpha t}{4}}\|\nabla^2(\varrho_0, d_0, \Omega_0, \varphi_0)\|_{L^2} + C\int_0^\frac t2e^{-\frac{\alpha}{4} (t-\tau)}\|\nabla^2(N_1(\tau), N_2(\tau))\|_{L^2}d\tau
  \\
 &
 \le C(\mathcal{E}_0, K_0)(1+t)^{-3}.
 \end{split}
 \end{equation}

Next, taking the summation $\eqref{es-varrho-H-3} +\eqref{es-u-H-3} +\frac{\beta}{\mu}\times \eqref{es-u-varphi-H-3} +\frac{\beta P'(\bar\rho)}{\mu\gamma\bar\rho}\times\eqref{es-varphi-H-3} +K_4\eqref{es-varrho-u-H-3}$ with $K_4 = \min\left\{\frac{\alpha}{4\bar\rho}, \frac{P'(\bar\rho)}{4\mu\beta\bar\rho}\left(\frac{\nu P'(\bar\rho)}{\gamma \bar\rho} -\beta\right)\right\}$ gives rise to
  \begin{align}\label{de-high-u1}
   &\frac{d}{dt}\left(\left\langle \frac{P'(\varrho +\bar\rho)}{2(\varrho +\bar\rho)^2}\nabla^3\varrho^H, \nabla^3\varrho^H\right\rangle
   +\frac12\|\nabla^3 {\bf u}^H\|_{L^2}^2
   + \frac\beta\mu \langle\nabla^2 {\bf u}^H, \nabla^3\varphi^H\rangle
   +\frac{\beta P'(\bar\rho)}{2\mu\gamma\bar\rho}\|\nabla^3\varphi^H\|_{L^2}^2\right.
   \notag\\
   &\quad+
   K_4\langle\nabla^3\varrho^H, \nabla^2 {\bf u}^H\rangle
  \bigg)
  + \frac\alpha2\|\nabla^3 {\bf u}^H\|_{L^2}^2
  +\frac{\beta P'(\bar\rho)}{\gamma\bar\rho}\|\nabla^4\varphi^H\|_{L^2}^2
  \\
  &\quad +\frac{\beta}{4\mu}\left(\frac{\nu P'(\bar\rho)}{\gamma\bar\rho} - \beta\right)\|\nabla^3\varphi^H\|_{L^2}^2
  +\frac{K_4P'(\bar\rho)}{4\bar\rho}\|\nabla^3\varrho^H\|_{L^2}^2
  \notag\\
  &\le C\|\nabla^2 {\bf u}^H\|_{L^2}^2 + C(\|(\varrho, {\bf u})\|_{L^\infty} +\|\nabla(\varrho, {\bf u})\|_{L^3\cap L^\infty})\|\nabla^3(\varrho, {\bf u}, \varphi)\|_{L^2}^2.
  \notag
  \end{align}

  Define
  \begin{equation}\label{L3}
  \begin{split}
   \mathcal{L}_3(t)
   &= \left\langle \frac{P'(\varrho +\bar\rho)}{2(\varrho +\bar\rho)^2}\nabla^3\varrho^H, \nabla^3\varrho^H\right\rangle
   +\frac12\|\nabla^3 {\bf u}^H\|_{L^2}^2
   + \frac\beta\mu \langle\nabla^2 {\bf u}^H, \nabla^3\varphi^H\rangle
   +\frac{\beta P'(\bar\rho)}{2\mu\gamma\bar\rho}\|\nabla^3\varphi^H\|_{L^2}^2
   \\
   &\quad+
   K_4\langle\nabla^3\varrho^H, \nabla^2 {\bf u}^H\rangle.
   \end{split}
  \end{equation}
  We can get the following equivalent relationship:
 \begin{equation}\label{L3-u-2}
  \mathcal{L}_3(t) +C_5\|\nabla^2 {\bf u}^H(t)\|_{L^2}^2
  \approx \|\nabla^3(\varrho^H, {\bf u}^H, \varphi^H)\|_{L^2}^2 + \|\nabla^2 {\bf u}^H(t)\|_{L^2}^2.
 \end{equation}
Moreover, we have
 \begin{equation}\nonumber
  \frac{d}{dt}\mathcal{L}_3(t) +C_6\mathcal{L}_3(t)
  \le C\|\nabla^2 {\bf u}^H(t)\|_{L^2}^2 + C(\|(\varrho, {\bf u})\|_{L^\infty} +\|\nabla(\varrho, {\bf u})\|_{L^3\cap L^\infty})\|\nabla^3(\varrho, {\bf u}, \varphi)\|_{L^2}^2.
 \end{equation}
 Then by the Gronwall inequality and Lemma \ref{s1s2}, we can arrive at
 \begin{equation}\nonumber
  \begin{split}
  \mathcal{L}_3(t)
  &\le e^{-C_6t}\mathcal{L}_3(0) + C\int_0^t e^{-C_6(t-\tau)}\|\nabla^2 {\bf u}^H(\tau)\|_{L^2}^2d\tau
  \\
  &
  \quad + C\int_0^t e^{-C_6(t-\tau)}(\|(\varrho, {\bf u})(\tau)\|_{L^\infty} +\|\nabla(\varrho, {\bf u})(\tau)\|_{L^3\cap L^\infty})\|\nabla^3(\varrho, {\bf u}, \varphi)(\tau)\|_{L^2}^2d\tau
  \\
  &
  \le e^{-C_6t}\mathcal{E}_0^2 + C\int_0^t e^{-C_6(t-\tau)}C^2(\mathcal{E}_0, K_0)(1+t)^{-6}d\tau
  \\
  &
  \le C(\mathcal{E}_0, K_0)(1+ t)^{-6},
  \end{split}
 \end{equation}
 which together with the relationship \eqref{L3-u-2} yields \eqref{e-u-3-high}.
 \end{proof}

 Proposition \ref{pro-de-high-u} can be deduced immediately by combining with Lemma \ref{le-de-u-3-low} and Lemma \ref{le-de-u-3-high}.

\section{Analytic tools}\label{S6}\label{S6}

We will extensively use the Sobolev interpolation of the Gagliardo--Nirenberg inequality; the proof can be seen in \cite{N3}.
\begin{lemma}\label{1interpolation}
 Let $0\le i, j\le k$, then we have
 \begin{equation}\nonumber
  \|\nabla^i f\|_{L^p}\lesssim \|\nabla^jf\|_{L^q}^{1-a}\| \nabla^k f\|_{L^r}^a,
 \end{equation}
 where $a$ belongs to $\left[\frac ik, 1\right]$ and satisfies
 \begin{equation}\nonumber
  \frac{i}{3}-\frac{1}{p}= \left(\frac{j}{3}-\frac{1}{q}\right)(1-a)+ \left(\frac{k}{3}-\frac{1}{r}\right)a.
 \end{equation}

 Especially, while $p=q=r=2$, we have
 \begin{equation}\nonumber
  \|\nabla^if\|_{L^2}\lesssim \|\nabla^jf\|_{L^2}^\frac{k-i}{k-j}
  \|\nabla^kf\|_{L^2}^\frac{i-j}{k-j}.
 \end{equation}
\end{lemma}

To estimate the product of two functions, we shall record the following estimate, cf. \cite{J}:
\begin{lemma}\label{es-product}
 It holds that for $k\geq0$,
 \begin{equation}\nonumber
  \|\nabla ^k(gh)\|_{L^{p_0}} \lesssim \|g\|_{L^{p_1}}\|\nabla^kh\|_{L^{p_2}} +\|\nabla^kg\|_{L^{p_3}}\|h\|_{L^{p_4}}.
 \end{equation}
 Here $p_0, p_2, p_3\in (1, \infty)$ and
 \begin{equation}\nonumber
  \frac1{p_0} = \frac1{p_1} +\frac1{p_2} = \frac1{p_3} +\frac1{p_4}.
 \end{equation}
\end{lemma}

Thus we can easily deduce from Lemma \ref{es-product} the following commutator estimate:
\begin{lemma}\label{1commutator}
 Let $f$ and $g$ be smooth functions belonging to $H^k\cap L^\infty$ for any integer $k\ge1$ and  define the commutator
 \begin{equation}\nonumber
  [\nabla ^k,f]g=\nabla ^k(fg)-f\nabla ^kg.
 \end{equation}
 Then we have
 \begin{equation}\nonumber
  \|[\nabla ^k,f]g\|_{L^{p_0}} \lesssim \|\nabla  f\|_{L^{p_1}}\|\nabla ^{k-1}g\|_{L^{p_2}}+\|\nabla ^k f\|_{L^{p_3}}\| g\|_{L^{p_4}}.
 \end{equation}
 Here $p_i~(i=0,1,2,3,4)$ are defined in Lemma \ref{es-product}.
\end{lemma}

Next, to estimate the $L^2$--norm of the spatial derivatives of some smooth function $F(f)$, we shall introduce some estimates which follow from Lemma \ref{1interpolation} and Lemma \ref{es-product}:
\begin{lemma}\label{infty}
 Let $F(f)$ be a smooth function of $f$ with bounded derivatives of any order and $f$ belong to $H^k$ for any integer $k\ge3$ , then we have
 \begin{equation}\nonumber
  \|\nabla ^k(F(f))\|_{L^2} \lesssim \sup_{0\le i\le k}\|F^{(i)}(f)\|_{L^\infty}
  \left(\sum_{j=2}^k\|f\|_{L^2}^{j-1-\frac{3(j-1)}{2k}}\|\nabla ^kf\|_{L^2}^{1+\frac{3(j-1)}{2k}}+ \|\nabla ^kf\|_{L^2}\right).
 \end{equation}

 Moreover, if $f$ has the lower and upper bounds, and $\|f\|_k\le 1$, we have
 \begin{equation}\nonumber
  \|\nabla ^k(F(f))\|_{L^2} \lesssim \|\nabla ^kf\|_{L^2}.
 \end{equation}
\end{lemma}

\subsection*{Acknowledgments}
Qing Chen's research is supported in part by National Natural Science Foundation of China
(No. 12271114), National Natural Science Foundation of China (No. 12171401) and National Science Foundation of Fujian Province, China (No. 2022J011241). Guochun Wu's research is supported in part by National Science Foundation of Fujian Province, China (No. 2022J01304). H. Wang's research is supported the National Natural Science Foundation of China (No. 11901066), the Natural Science Foundation of Chongqing (No. cstc2019jcyj-msxmX0167) and projects Nos. 2022CDJXY-001, 2020CDJQY-A040 supported by the Fundamental Research Funds for the Central Universities.

\end{document}